\documentclass[a4paper,12pt]{article}
\usepackage[utf8]{inputenc}
\usepackage[T1]{fontenc}
\usepackage[english]{babel}
\usepackage{amsmath}
\usepackage{amssymb}
\usepackage{amsthm}
\usepackage{textcomp}
\usepackage[mathcal]{euscript}
\usepackage{graphicx}
\usepackage[top=3cm, bottom=3cm, left=3cm, right=3cm]{geometry}
\usepackage{subcaption}
\usepackage{afterpage}
\usepackage{pdflscape}
\usepackage[affil-it]{authblk}
\usepackage[linesnumbered, lined, boxed, commentsnumbered]{algorithm2e}
\usepackage[round]{natbib}

\usepackage[colorlinks=true,urlcolor=blue,linkcolor=blue, citecolor=red]{hyperref}
\theoremstyle{definition}
\newtheorem{mydef}{Definition}

\theoremstyle{plain}
\newtheorem{theo}[mydef]{Theorem}

\newtheorem{prop}[mydef]{Proposition}
\newtheorem{problem}{Problem}

\theoremstyle{remark}
\newtheorem{rem}[mydef]{Remark}

\def\mathscr{\EuScript}
\newcommand{\probelement}{q}
\newcommand{\probset}{\cQ}
\newcommand{\phidiv}[2]{D_{\varphi}\np{#1,#2}}

\newcommand{\cH}{\mathcal{H}}
\newcommand{\cM}{\mathcal{M}}
\newcommand{\cQ}{\mathcal{Q}}
\newcommand{\BB}{\mathbb{B}}
\newcommand{\NN}{\mathbb{N}}
\newcommand{\RR}{\mathbb{R}}
\newcommand{\WW}{\mathbb{W}}
\newcommand{\XX}{\mathbb{X}}
\newcommand{\abs}[1]{\left|#1\right|}                       % Valeur absolue
\newcommand{\np}[1]{(#1)}                                   % Parenth\`{e}se normal
\newcommand{\bp}[1]{\big(#1\big)}                           % Parenth\`{e}se big
\newcommand{\Bp}[1]{\Big(#1\Big)}                           % Parenth\`{e}se Big
\newcommand{\Bgp}[1]{\Bigg(#1\Bigg)}                        % Parenth\`{e}se Bigg
\newcommand{\nc}[1]{[#1]}                                   % Crochet normal
\newcommand{\bc}[1]{\big[#1\big]}                           % Crochet big
\newcommand{\Bc}[1]{\Big[#1\Big]}                           % Crochet Big
\newcommand{\nce}[1]{[\![#1]\!]}                                       % Crochet entier normal
\newcommand{\na}[1]{\{#1\}}                                 % Accolade normal
\newcommand{\ba}[1]{\big\{#1\big\}}                         % Accolade big
\newcommand{\Ba}[1]{\Big\{#1\Big\}}                         % Accolade Big
\newcommand{\Bga}[1]{\Bigg\{#1\Bigg\}}                      % Accolade Bigg
\newcommand{\argmin}{\mathop{\arg\min}}                     % Arg-min
\newcommand{\tribu}[1]{\mathscr{#1}}                        % Tribu
\newcommand{\omeg}{\Omega}                                  % espace du triplet
\newcommand{\trib}{\tribu{F}}                               % tribu  du triplet
\newcommand{\prbt}{p}                              % proba  du triplet
\newcommand{\espe}{\mathbb{E}}                              % Symbole esp\'{e}rance
\newcommand{\mesr}{\mathbb{F}}                              % Symbole mesure de risque
\makeatletter
\def\va@a{\boldsymbol{\va@arg^{\textstyle\text{\unboldmath$\scriptstyle\va@expo$}}_{\textstyle\text{\unboldmath$\scriptstyle\va@index$}}}}
\def\va#1{\def\va@expo{}\def\va@index{}\def\va@arg{\uppercase{#1}}%
  \@ifnextchar^{\va@h}{\@ifnextchar_\va@u\va@a}}
\def\va@h^#1{\def\va@expo{#1}\@ifnextchar_\va@hu\va@a}
\def\va@u_#1{\def\va@index{#1}\@ifnextchar^\va@uh\va@a}
\def\va@hu_#1{\def\va@index{#1}\va@a}
\def\va@uh^#1{\def\va@expo{#1}\va@a}
\makeatother
\newcommand{\normdelim}[1]{\nc{#1}}                         % Taille ``normal''
\newcommand{\bigdelim}[1]{\bc{#1}}                          % Taille ``big''
\newcommand{\Bigdelim}[1]{\Bc{#1}}                          % Taille ``Big''
\newcommand{\vardelim}[1]{\left(#1\right)}                  % Taille ``variable''
\newcommand{\nesp}[2]{\espe_{#1}\normdelim{#2}}           % Esp\'{e}rance normal
\newcommand{\nmes}[2]{\mesr_{#1} \normdelim{#2}}           % Mesure de risque normal
\newcommand{\bmes}[2]{\mesr_{#1} \bigdelim{#2}}            % Mesure de risque big
\newcommand{\Bmes}[2]{\mesr_{#1} \Bigdelim{#2}}            % Mesure de risque Big
\newcommand{\proscal}[2]{\left\langle#1\:,#2\right\rangle}  % Produit scalaire
\newcommand{\nnorm}[1]{\|#1\|}                              % Norme
\newcommand{\dom}{\mathop{\mathrm{dom}}}                    % Domaine
\def\eqsepv{\; , \enspace}                                  % Virgule dans une \'{e}quation
\def\eqfinv{\; ,}                                           % Virgule en fin d'\'{e}quation
\def\eqfinp{\; .}                                           % Point en fin d'\'{e}quation
\title{General Risk Measures\\for Robust Machine Learning}
\author[ ]{Émilie Chouzenoux\textsubscript{a}\thanks{emilie.chouzenoux@centralesupelec.fr}}
\author[ ]{Henri Gérard\textsubscript{b}\thanks{hgerard.pro@gmail.com}}
\author[ ]{Jean-Christophe Pesquet\textsubscript{a}\thanks{jean-christophe@pesquet.eu}}
\affil[ ]{\textsubscript{a}\small \it CentraleSup\'elec, Inria Saclay, Universit\'e Paris-Saclay
\\
Center for Visual Computing\\
Gif sur Yvette, 91190, France}
\affil[ ]{\textsubscript{b}\small \it Université Paris-Est, CERMICS (ENPC), Labex Bézout\\
6-8 avenue Blaise Pascal\\
Champs-sur-Marne, 77420, France}
\date{}

\begin{document}

\maketitle

\begin{abstract}
A wide array of machine learning problems are formulated as the minimization of the expectation of a convex loss function
on some parameter space. Since the probability distribution of the data of interest is usually unknown, it is 
is often estimated from training sets, which may lead to poor out-of-sample performance. 
In this work, we bring new insights in this problem by using the framework which has been developed in quantitative finance for risk measures. We show that the original min-max problem can be recast as
a convex minimization problem under suitable assumptions. We discuss several important examples of
robust formulations, in particular by defining ambiguity sets based on $\varphi$-divergences and the Wasserstein metric.
We also propose an efficient algorithm for solving the corresponding convex optimization problems involving 
complex convex constraints. Through simulation examples, we demonstrate that this algorithm scales well
on real data sets.
%
%    Robust optimization aims at ameliorate the behavior
%    of solutions facing new data by formulating saddle point problems.
%    A common way to tackle such min-max problem is to 
%    reformulate them as a convex problem.
%    We bring here a new insight in robust optimization
%    making connections with financial mathematics.
%    We state a common convex formulation that 
%    groups together formulations based on ambiguity sets
%    defined by $\varphi$-divergence and Wasserstein distance.
%    Our main contribution is to provide an algorithm
%    for the general statement that scales on big instances. 
\end{abstract}

\textbf{Keywords:} Risk measures, robust statistics, machine learning, convex optimization, divergences, Wasserstein distance.
% \com{Programme de travail
%     \begin{itemize}
% %        \item Lire papier projection sous différentiel
% %        \item Poursuivre la rédaction
% %        \item écrire le problème sous forme de minimisation d'un coût linéaire sous contrainte
% %        \item Implémenter en julia l'algo de projection (dans un premier temps)
% %        \item Vérifier que juste la projection marche bien car le coût doit augmenter
%     \end{itemize}
% }

% \tableofcontents

% \com{citer des articles en introduction poru
%     \begin{itemize}
%        \item Robustesse
%        \item Divergence
%        \item Mesures de risques
%        \item Distance de Wasserstein et transport optimal
%        \item Proximal algorithms
%     \end{itemize}
% }
% \newpage
\section{Introduction}
In machine learning,  the robustness of the solutions obtained for classification 
and prediction tasks remains a main issue.
%Facing a machine learning problem,
%where you have to classify some data
%or predict some quantities,
%an issue remains the robustness of solutions.
%No bounds on the error are in general provided
In \citet*{papernot2016transferability} and~\citet*{kurakin2016adversarial},
some examples are provided  where small modifications of the input data
can completely alter the resulting solution. 
In \citet*{feng2014robust} and~\citet*{plan2013robust}, poor out-of-sample
performances are displayed when training data is sparse.
This kind of problems also occurs in optimal control
when there exist uncertainties on parameters. In~\citet*{ben2000robust},
the authors showed that a small perturbation on the parameters
can turn a feasible solution into an infeasible one.

In this context, robust approaches appear as a way of
controlling out-of-sample performance. There is an extensive literature
dealing with robust problems and the reader is refered to~\citet*{ben2009robust}
for a survey.
One of the main approaches consists of introducing  constraints on the probability distribution of the unknown data.
Under some conditions, this approach is equivalent
to deal with ambiguity sets or a modified loss function.
The works in 
\citet*{ben2013robust,hu2013kullback,duchi2016statistics,moghaddam2016robust}
and \citet*{namkoong2016stochastic} have brought more insight
on ambiguity sets. In~\citet*{esfahani2015data} and~\citet*{esfahani2017data},
the authors present a distributionally robust optimization framework based on the
Wasserstein distance. A set of probability distributions is defined
as a ball centered on the reference probability with respect to the
Wasserstein distance, then the optimization is carried out for the worst cost over this
probability set.

This idea of minimizing the worst cost over a given probability set is well-known
in quantitative finance. The robust representation of risk measures
provides a theoretical framework to do so. A rich class of risk measures
is the class of coherent ones which
were introduced in the seminal paper by~\citet*{artzner1999coherent}.
In \citet*{follmer2016stochastic}, a broader class of so-called convex risk measures was investigated, 
for which a large number of results were established.

% \jcp{Si on soumet \`a IEEE TSP il faudrait faire r\'ef\'erence aux travaux de Yonina Eldar sur les approches minimax.}

In this paper, we follow the line of~\citet*{esfahani2015data},
which aims at reformulating robust problems using ambiguity sets as 
convex minimization problems.
% Several ambiguity sets have been studied
% independently.  \citet*{??}, \citet*{??} and \citet*{??}
% focus on ambiguity defined by $\varphi$ divergences risk measures
% \citet*{??} chooses ambiguity sets that are defined as a Wasserstein ball.
Our contribution is threefold. First we clarify the links
existing between risk measures and robust optimization.
This allows us to transpose results from finance to machine learning.
Second, we propose a unifying convex optimization setting for dealing
with various risk measures, including those based on $\varphi$-divergences or the Wasserstein distance.
Finally, we propose an accelerated algorithm grounded on the subgradient projection method proposed
in \citet*{combettes2003block}. We show that the proposed algorithm is able to solve efficiently large-scale robust problems.

The organization of the paper is as follows.
In Section~\ref{section_problem_statement}, we state the general mathematical problem
we investigate in the context of machine learning.
In Section~\ref{section_risk_measures}, we first draw a parallel
between this problem and convex monetary risk measures. 
We then provide a convex reformulation of the problem.
In Section~\ref{section_literature}, we discuss some important classes of risk measures 
by revisiting some of the results in the literature.
%revisit the literature
%highlighting interesting cases.
In Section~\ref{section_numerical_experiments},
we describe our algorithm for solving convex formulations of robust problems.
Then, in Section~\ref{se:sim}, we illustrate the good performance of the proposed algorithm through numerical experiments
on real datasets. Finally, short concluding remarks are made in Section \ref{se:conclu}.

% \com{Pourquoi étudie-t-on la robustesse ?}
% \com{Comment étudier : plusieurs axes avec notament, par contraintes, dans l'objectif, ..}
% \com{On choisit l'axe détude de Daniel Kuhn et autres auteurs qui consistent à reformuler en un problème convexe}
% \com{Plusieurs ambiguity sets ont été étudié de manière indépendante}
% \com{Notre contribution, faire le lien avec les mesures de risques et proposer une reformulation unificatrice}
% \com{Puis donner une interprétation de cette reformulation}
% \com{Enfin, fournir des algorithmes capables de scaler à grande échelle pour résoudre les problèmes robustes}

\section{Problem statement}\label{section_problem_statement}
% \com{Ne pas introduire p tout de suite}
% \com{Mettre $\omega_{i}$}
Let $\np{\omeg, \trib, \prbt}$ be the underlying probability space
where $\omeg$ is a finite set of cardinal $N$,
$\trib$ is the $\sigma$-field generated by $\np{\na{\omega}}_{\omega \in \omeg}$,
and $\prbt$ is a probability distribution that is assumed to charge all points.
Let $d$ be a nonzero integer and let $\mathbf{z} \colon \omeg \to \RR^{d}$ denote a general random variable.
Note that function $\mathbf{z}$ can be identified with a matrix in $\RR^{N\times d}$
where, for every $i \in \nce{1,N}$, $z_{i}\in \RR^{d}$ is the vector corresponding to the $i$-th line of matrix $\mathbf{z}$ .
We denote by $\cM_{1}$ the set of probability distributions over $\np{\omeg, \trib, \prbt}$.

For every  $i \in \nce{1,N}$, let $\ell(\cdot,z_{i})\colon \RR^n \to \RR \cup \na{+\infty}$ be a loss function which is assumed to
be lower semicontinuous (lsc) and convex such that
\begin{equation}\label{eq_loss_prop}
\bigcap_{i=1}^{N}\dom\big(\ell(\cdot,z_{i})\big) \neq \varnothing\eqfinv
\end{equation}
where $\dom(g)$ denotes the domain of a function $g$, that is the set of argument values for which this function is finite.
In standard formulations of machine learning problems, one aims at finding an optimal regression vector $\overline{\theta} \in \RR^n$ such that
\begin{equation}\label{e:classform}
\overline{\theta} \in \argmin_{\theta \in \RR^{n}}
        \sum_{i = 1}^{N}
        \prbt_{i}
        \ell \np{\theta, z_{i}}
    \eqfinp
\end{equation}
Indeed, setting $z_{i} = [x_{i}^\top\;\; y_{i}]^{\top}$ with $n=d-1\ge 1$, $x_{i} \in \RR^{n}$,
and $y_{i}\in \RR$ allows us to recover  a wide array of estimation and classification problems.
For example, penalized least squares regression problems are obtained when
    \begin{equation}\label{eq_penalized_loss_function}
    (\forall i \in\nce{1,N})(\forall \theta\in \RR^{n})\quad
        \ell\np{\theta,z_{i}} =
        \frac{1}{2} \nnorm{y_i - x_{i}^{\top} \theta}^2 + \rho\np{\theta}
        \eqfinv
    \end{equation}
where $\rho\colon \RR^{n} \to \RR \cup \na{+\infty}$ is a proper, lsc, convex penalty function.
If the random variable $\mathbf{y}$ is $\na{0,1}$-valued, we can recover binary classification problems, 
for example by performing a logistic regression, i.e.
    \begin{equation}\label{e:logreg}
     (\forall i \in\nce{1,N})(\forall \theta\in \RR^{n})\quad
           \ell\np{\theta, z_i}
            =
            \log\Bp{
                1+\exp\bp{-y_{i}x_{i}^{\top} \theta}
            }
        \eqfinp
    \end{equation}
 One of the main limitations of this formulation is that it assumes that the true probability distribution of the data is perfectly known.
 In practice, this distribution is often estimated empirically from the available observations. 

In this paper, we will focus on the following more general robust formulation to 
determine  an optimal regression vector.
\begin{problem}  
\label{initial_problem_generic}
Let $\alpha\colon \RR^N \to \RR \cup \na{+\infty}$ be a lsc convex
penalty function whose domain is a nonempty subset of $\cM_{1}$.
We want to find
\begin{equation}\label{eq_initial_problem_generic}
    \overline{\theta} \in \argmin_{\theta \in \RR^{n}}
    \sup_{\probelement= (\probelement_{i})_{1\le i \le N} \in \cM_{1}}
    \vardelim{
        \sum_{i = 1}^{N}
        q_{i}
        \ell \np{\theta, z_{i}}
        -\alpha\np{\probelement}
    }
    \eqfinv
%      \tag{GRP}
 \end{equation}
\end{problem}

%Let us detail some examples to motivate the study of Problem~\eqref{initial_problem_generic}.
In this problem, if $\alpha$  is the indicator function $\iota_{\na{\prbt}}$
of the singleton containing a probability distribution $\prbt$,
then \eqref{e:classform} is recovered.\footnote{The indicator of a set is the function equal to
0 on this set and $+\infty$ otherwise.} 
More generally, if $\alpha$ is
equal to the indicator function of a nonempty closed convex set $\probset \subset \cM_{1}$,
then the objective function in \eqref{eq_initial_problem_generic} reduces to
\begin{equation}
    \sup_{q\in \probset} \sum_{i=1}^N q_i \ell\np{\theta,z_i}
    =
    \sigma_{\probset}\bp{\mathbf{\ell}\np{\theta,\mathbf{z}}}
    \eqfinv
\end{equation}
where $\sigma_{\probset}$ is the support function of $\probset$. 
This corresponds to the well-known case
of distributionally robust optimization using ambiguity sets \cite{esfahani2015data}.

\section{Convex formulation of robust inference problems using risk measures}\label{section_risk_measures}
In this section, we address Problem~\ref{initial_problem_generic}
in the light of the financial framework for monetary risk measures.
We first recall known properties of risk measures and 
then show how Problem~\ref{initial_problem_generic}
can be reformulated as a convex problem.

\subsection{Definition and properties of a risk measure}
Let $\XX$ be the space of real-valued random variables
defined on the probability space $\np{\omeg,\trib,\prbt}$.
We denote by $\va{X}$ a generic element of $\XX$ and
we recall that $\prbt$ is assumed to be a distribution that charges all points.
The space $\XX$ is endowed with the pointwise order $\leq$, that is,
\begin{equation}\label{eq_def_order}
(\forall (\va{X},\va{Y})\in \XX^2)\quad
    \va{X} \leq \va{Y}
    \Leftrightarrow (\forall \omega \in \omeg)\;\;
    \va{X}\np{\omega} \leq \va{Y}\np{\omega}    
    \eqfinp
\end{equation}
A \emph{risk measure} $\mesr$ is a real-valued function
$
    \mesr : \XX \to \RR
    \eqfinp
$\\
The next four properties of risk measures were first introduced in~\citet*{artzner1999coherent}
to define the so called \emph{coherent risk measures}. The interested reader can also
refer to~\citet*{follmer2016stochastic}[Part I, Chapter 4].

\begin{mydef}
    A risk measure $\mesr \colon \XX \to \RR$ is said to be
    \begin{itemize}
        \item \emph{monotone:} if, for every $(\va{X},\va{Y})\in \XX^2$, $\va{X} \leq \va{Y} \Rightarrow \nmes{}{\va{X}} \leq \nmes{}{\va{Y}}$,
        \item \emph{translation invariant:} if, for every $\va{X}\in \XX$ and $m \in \RR$, $\nmes{}{\va{X} + m} = \nmes{}{\va{X}} + m$,
%         \item \emph{normalization:} $\nmes{}{0} = 0$,
        \item \emph{convex:} if, for every $(\va{X},\va{Y})\in \XX^2$ and $\lambda \in ]0,1[$, $\bmes{}{\lambda \va{X} + \np{1-\lambda}\va{Y}} \leq \lambda\nmes{}{\va{X}} + \np{1-\lambda}\nmes{}{\va{Y}}$,
        \item \emph{positively homogeneous:} if, for every $\va{X}\in \XX$ and $\lambda \in [0,+\infty[$,
        $\nmes{}{\lambda \va{X}} = \lambda\nmes{}{\va{X}}$.
    \end{itemize}
    A risk measure which satisfies the first two properties is called a \emph{monetary risk measure}.
    A risk measure which satisfies the first three properties is called a \emph{convex risk measure}.
    A risk measure which satisfies the four properties is called a \emph{coherent risk measure}.
\end{mydef}

Depending on the author, the first axiom may also be expressed as:
for every $(\va{X},\va{Y})\in \XX^2$,
$\va{X} \leq \va{Y} \Rightarrow \nmes{}{\va{X}} \geq \nmes{}{\va{Y}}$
if the variables $\va{X}$ and $\va{Y}$ are interpreted as gains instead of a losses, 
which is a common in finance.
For this reason, some sign differences may appear between results of various authors.
We have chosen to follow the paths
in~\cite{rockafellar2000optimization,ruszczynski2006optimization,ruszczynski2006conditional}
and interpret the random variable in argument as a loss.
We however often refer to~\cite{follmer2016stochastic}, providing a comprehensive view of risk measures,
where the opposite convention has been adopted.

\begin{rem}\  \label{re:risk}
\begin{enumerate}
 \item\label{re:riski}    It readily follows from the translation invariance property that a monetary risk measure $\mesr_{}$ admits a primal
    form given by
    \begin{equation}\label{eq_repres_acceptance}
    (\forall \va{X}\in \XX)\quad
        \nmes{}{\va{X}}
        =
        \inf_{s \in \RR}
        \na{s \mid \va{X} -s \in \operatorname{lev}_{\le 0} \mesr_{}}
%        \cA_{\mesr_{}}}
        \eqfinv
    \end{equation}
    where 
    %$\cA_{\mesr_{\alpha}}$
    $\operatorname{lev}_{\le 0} \mesr_{}$ is the lower level set of $\mesr_{}$ at height 0 defined as
       \begin{equation}\label{e:deflevset0}
        %\cA_{\mesr_{}}
        \operatorname{lev}_{\le 0} \mesr_{}
        =
        \ba{\va{X} \in \XX \mid \bmes{}{\va{X}} \leq 0}
        \eqfinp
    \end{equation}
\item\label{re:riskii}  
    A monetary risk measure $\mesr$ is $1$-Lipschitz continuous
    with respect to the supremum norm $\nnorm{\cdot}_{\infty}$. Indeed, 
    for every $(\va{X},\va{Y})\in \XX^2$,
    we have
    $\va{X} \leq \va{Y} + \nnorm{\va{X} - \va{Y}}$. By monotonicity and translation invariance
    we obtain that $\nmes{}{\va{X}} -\nmes{}{\va{Y}} \leq \nnorm{\va{X} - \va{Y}}_{\infty}$,
    which by symmetry implies that $|\nmes{}{\va{X}} -\nmes{}{\va{Y}}| \leq \nnorm{\va{X} - \va{Y}}_{\infty}$.
    \end{enumerate}
\end{rem}

The class of convex risk measures includes a large number of useful functions.
Without entering into details, we should mention: expectation, worst case,
quantile, median, and average value at risk \cite{follmer2016stochastic}.
% \com{Pinball loss}

\subsection{Convex reformulation}
In this section, we will show that the ``min-max'' problem~\ref{initial_problem_generic}
admits a convex reformulation.
We first gather in the following proposition some existing results in the literature.
\begin{prop}\label{prop:preconvex}
$\mesr{}$ is a convex risk measure if and only if there exists a lsc and convex function
$\alpha\colon \RR^N \to \RR \cup \na{+\infty}$ such that
    \begin{equation}
    (\forall \va{X} \in \XX)\quad
        \nmes{}{\va{X}}
        =
        \sup_{\probelement \in \cM_{1}}\left(
        \sum_{i = 1}^{N}
        \probelement_{i}
        x_{i}
        - \alpha\np{\probelement}\right)
        \eqfinp
    \end{equation}
    The function $\alpha$ associated with $\mesr{}$ is uniquely defined as 
            \begin{equation} \label{eq_def_alpha}
           (\forall \probelement \in \RR^{N}) \quad \alpha\np{\probelement}
             =
             \begin{cases}
            \displaystyle\sup_{\va{X} \in \operatorname{lev}_{\le 0} \mesr_{}
            %\cA_{\mesr}
            }
            \nesp{\probelement}{\va{X}} & \mbox{if $\probelement \in \cM_{1}\eqfinv$}\\
            +\infty & \mbox{otherwise}\eqfinp
            \end{cases}
        \end{equation}
      In addition, $\mesr{}$ is coherent if and only if its conjugate function $\alpha$ is the indicator function of a nonempty closed convex subset of $\cM_{1}$.

\end{prop}

\begin{proof}\ 
\begin{enumerate}
\item
        We know from \citet*[Theorem 4.16 and Proposition 4.15]{follmer2016stochastic}
        that any convex risk measure $\mesr$ on $\XX$ is of the form
        \begin{equation}\label{eq_def_conve_alpha}
        (\forall \va{X} \in \XX)\quad
            \nmes{}{\va{X}}
            =
            \sup_{\probelement \in \cM_{1}}
            \bp{
                \nesp{\probelement}{\va{X}}
                - \alpha\np{\probelement}
            }
            \eqfinv
        \end{equation}
        where $\alpha\colon \RR^N \to \RR \cup \na{+\infty}$ is the lsc and convex function
        whose domain is a nonempty subset of $\cM_{1}$,  given by
        \begin{align}
           (\forall \probelement \in \cM_{1}) \quad \alpha\np{\probelement}
            &=
            \sup_{\va{X} \in \XX}
            \nesp{\probelement}{\va{X}} - \nmes{}{\va{X}}
            \eqfinv\nonumber\\
            &=
            \sup_{\va{X} \in \operatorname{lev}_{\le 0} \mesr_{}
            %\cA_{\mesr}
            }
            \nesp{\probelement}{\va{X}}
        \end{align}
(the second equality stems from Remark \ref{re:risk}\ref{re:riski}).

       Conversely, one can associate to every lsc convex function
$\alpha\colon \RR^N \to \RR \cup \na{+\infty}$ whose domain is a nonempty subset of $\cM_{1}$
        a unique convex risk measure defined by~\eqref{eq_def_conve_alpha}.
        
 \item               It follows from \citet*[Proposition 4.15]{follmer2016stochastic}  that if, in addition,
        the risk measure $\mesr$ is coherent, then the function $\alpha$ in~\eqref{eq_def_alpha}
        is the indicator function of a nonempty closed convex subset of $\cM_{1}$ and the converse property holds.
\end{enumerate}
        
\end{proof}

We now state the main result of this section.
\begin{theo}\label{theo_representation}
Let $\alpha\colon \RR^N \to \RR \cup \na{+\infty}$ be a lsc convex function
        whose domain is a nonempty subset of $\cM_{1}$. Problem~\ref{initial_problem_generic} is equivalent to find
    \begin{equation}\label{eq_convex_infpre}
    \overline{\theta}\in
        \argmin_{\theta \in \RR^{n}}
        \bmes{\alpha}{\ell \np{\theta, \mathbf{z}}}
        \eqfinv
    \end{equation}
    where
    \begin{equation}\label{eq_dem_equivalence_problem1}
    (\forall \va{X} \in \XX)\quad
        \nmes{\alpha}{\va{X}}
        =
        \max_{\probelement \in \cM_{1}}\left(
        \sum_{i = 1}^{N}
        \probelement_{i}
        x_{i}
        - \alpha\np{\probelement}\right)
        \eqfinp
    \end{equation}
    The function $\bmes{\alpha}{\ell \np{\cdot, \mathbf{z}}}$ is proper, lsc, and convex.
    In addition, the so-defined convex optimization problem admits a primal formulation which consists of finding
    \begin{equation}\label{eq_convex_inf}
     \inf_{(\theta,s)\in \mathcal{S}} s
    \end{equation}
    where
    \begin{equation}\label{e:defS}
    \mathcal{S} = \{(\theta,s)\in \RR^{n}\times \RR \mid \ell \np{\theta, \mathbf{z}} - s \in 
    \operatorname{lev}_{\le 0} \mesr_{\alpha}
    %\cA_{\mesr_{\alpha}}
    \}
    \end{equation}
       \eqfinp
  \end{theo}
%\noindent 
%and  $\cA_{\mesr_{\alpha}}$ is given by  \eqref{e:deflevset0}.
\begin{proof}
It follows from Proposition \ref{prop:preconvex} that \eqref{eq_initial_problem_generic}  is equivalent to
\eqref{eq_convex_infpre}
        where $\mesr_{\alpha}$ is a convex risk measure. In addition, the $\sup$ in the definition of the risk measure is attained since $\cM_{1}$ is a compact set and $q \mapsto \sum_{i = 1}^{N} \probelement_{i} x_{i}- \alpha\np{\probelement}$ is upper semicontinuous.
        
        The function $\ell\np{\cdot, Z}$ is lsc convex
        for every $Z \in \RR^{d}$. Given a random variable $\mathbf{z}$, for every
        vectors $\theta_{1}$ and $\theta_{2}$ in $\RR^{n}$,
        and scalar $\lambda \in \nc{0,1}$, the convexity of function $\ell$ yields
        \begin{align}
        (\forall \omega\in \omeg)\quad
            \ell \np{\lambda \theta_{1} + \np{1-\lambda}\theta_{2}, \mathbf{z}\np{\omega}}
            \leq
            \lambda \ell \np{ \theta_{1}, \mathbf{z}\np{\omega}}+
            \np{1-\lambda} \ell \np{ \theta_{2}, \mathbf{z}\np{\omega}}
            \eqfinp
        \end{align}
        Now, by using the fact that the risk measure $\mesr_{\alpha}$ is monotone and convex
        with respect to the ordering introduced in~\eqref{eq_def_order}, we get
        \begin{align}
            \Bmes{\alpha}{\ell \np{\lambda \theta_{1} + \np{1-\lambda}\theta_{2}, \mathbf{z}}}
            &\leq
            \Bmes{\alpha}{\lambda \ell \np{ \theta_{1}, \mathbf{z}}
                +\np{1-\lambda} \ell \np{ \theta_{2}, \mathbf{z}}}
            \nonumber\\
            &\leq
            \lambda \Bmes{\alpha}{ \ell \np{ \theta_{1}, \mathbf{z}}}
                +\np{1-\lambda}\Bmes{\alpha}{\ell \np{ \theta_{2}, \mathbf{z}}}
            \eqfinp
        \end{align}
        This shows that $\bmes{\alpha}{\ell \np{\cdot, \mathbf{z}}}$ is convex.
        
        In addition, since $\mesr_{\alpha}{}$ is monotone and continuous (see Remark \ref{re:risk}\ref{re:riskii}) and
         $\ell \np{\cdot, \mathbf{z}}$ is lsc, $\bmes{\alpha}{\ell \np{\cdot, \mathbf{z}}}$ is 
         lsc. Because of \eqref{eq_loss_prop}, $\bmes{\alpha}{\ell \np{\cdot, \mathbf{z}}}$ is also
         proper.
         
         Finally, formulation \eqref{eq_convex_inf} is deduced from \eqref{eq_repres_acceptance} and \eqref{eq_convex_infpre}.
         
\end{proof}

The general convex reformulation~\eqref{eq_convex_inf} is not always easy
to handle. 
In practical applications, the choice of the mapping $\alpha$ plays a crucial role in this regard.
We will see in the next section some useful examples of this function.
 In particular, some mappings $\alpha$ lead to a formulation~\eqref{eq_convex_inf} that will be shown to be tractable numerically.

\section{Examples of risks measures}\label{section_literature}
By considering  particular forms of function $\alpha$ in~Problem \ref{initial_problem_generic},
we define three scenarios of interest for robust formulations.
The first two ones are based on $\varphi$-divergences, while the third one is based on the Wasserstein metric.
%we can specify the convex reformulation in ~\eqref{eq_convex_inf}.
%Hence we highlight the contribution of each author.
%We first provide a quick recall on divergence functions.
%Then we revisit three case well known in the literature.

\subsection{Perspective functions and divergences}

%We introduce the notion of $\varphi$-divergence.
% every probability distributions $\probelement \in \cM_{1}$ is absolutely continuous
% with respect to $\prbt$. Its Radon-Nikodym derivative $\frac{dp}{dq}$
% is equals to the vector of $\RR^{N}$: $\vardelim{\frac{\prbt_{i}}{\probelement_{i}}}_{i \in \nce{1,N}}$.
The notion of $\varphi$-divergence was first introduced independently
by~\cite{csiszar1964informationstheoretische,morimoto1963markov}
and~\cite{ali1966general}. For a more complete bibliography on the subject,
we refer to~\citet*{basseville2013divergence}.
%We recall that, in this work, $\omeg$ is finite and $\prbt$ charges all points.

%Here is the definition of a $\varphi$-divergence.
\begin{mydef}
    Let $\varphi: \RR \to ]-\infty, +\infty]$.
    The \emph{perspective function} $f_{\varphi}$ of function $\varphi$
    is given by
    \begin{align}
        f_{\varphi}: \RR \times \RR &\to ]-\infty, +\infty]
        \nonumber\\
        \np{x,\xi} &\mapsto
        \left \{
        \begin{array}{ll}
            \displaystyle \xi \varphi\vardelim{\frac{x}{\xi}}
            &\textrm{ if } \xi > 0
            \eqfinv
            \\
            +\infty
            &\textrm{ otherwise}
            \eqfinp
        \end{array}
        \right .
    \end{align}
\end{mydef}

\begin{mydef}\label{def_phi_divergence}
    Let $\varphi: \RR \to \nc{0, +\infty}$ be a lsc
    convex function with nonempty domain included in $[0,+\infty[$ such that $\varphi\np{1} = 0$.
    The \emph{$\varphi$-divergence}
    $D_{\varphi}\colon \RR^{N} \times \RR^{N} \to \nc{0, +\infty}$
    is defined as
    \begin{equation}
    \big(\forall \prbt = (\prbt_{i})_{1\le i \le N} \in \RR^{N}\big)
    \big(\forall  \probelement = (\probelement_{i})_{1\le i \le N}\in \RR^{N}\big)\quad
        \phidiv{\prbt}{\probelement}
        =
        \sum_{i=1}^{N}
        \underline{f_{\varphi}}
        \np{\prbt_{i},\probelement_{i}}
        \eqfinv
    \end{equation}
    where the function $\underline{f_{\varphi}}$ is the lsc envelope
    of the mapping $f_{\varphi}$, that is
    \begin{align}
        \underline{f_{\varphi}} : \RR\times \RR &\to ]-\infty, +\infty]
        \\
        \np{x,\xi} &\mapsto
        \left \{
        \begin{array}{ll}
            \displaystyle \xi \varphi\vardelim{\frac{x}{\xi}}
            &\textrm{ if } \xi > 0 \textrm{ and } x \geq 0
            \eqfinv
            \\
            \displaystyle x \lim_{t \to +\infty} \frac{\varphi\np{t}}{t}
            &\textrm{ if } \xi = 0 \textrm{ and } x > 0
            \eqfinv
            \\
            0
            &\textrm{ if } \xi = 0 \textrm{ and } x = 0
            \eqfinv
            \\
            +\infty
            &\textrm{ otherwise.}
        \end{array}
        \right .
    \end{align}
\end{mydef}

We also recall the definitions of a conjugate function and an adjoint function.
\begin{mydef}
    Let $\varphi: \RR \to ]-\infty, +\infty]$.
    The conjugate $\varphi^{*}$ of function $\varphi$ is defined by
    \begin{equation}
    (\forall s \in \RR)\quad
        \varphi^{*}\np{s}
        =
        \sup_{t \in \RR}
        \bp{ st - \varphi\np{t}}
        \eqfinv
    \end{equation}
    and the so-called adjoint function of $\varphi$ is defined by
    \begin{equation}
    (\forall t \in \RR)\quad
        \tilde{\varphi}\np{t}
        =
        \begin{cases}
         \displaystyle t \varphi\vardelim{\frac{1}{t}} & \mbox{if $t \geq 0$}\eqfinv\\
        \displaystyle \lim_{t\to +\infty} \frac{\varphi(t)}{t} & \mbox{if $t=0$\eqfinp}
        \end{cases}
    \end{equation}
\end{mydef}

Table~\ref{table_phi_divergence} is an extension of the one in~\citet*{ben2013robust}
and provides the expressions of common $\varphi$ functions, their conjugates, and the associated $\varphi$-divergence.
It is well-known \citep*{ben2013robust,combettes2018perspective}
that the adjoint $\tilde{\varphi}$ of $\varphi$ is such that
\begin{equation}
(\forall (\prbt,\probelement)\in 
(\RR^{N})^{2}) \quad
D_{\tilde{\varphi}}\np{\prbt, \probelement} = \phidiv{\probelement}{\prbt}
\end{equation}
and the conjugate of function $\lambda \varphi$
is
\begin{equation}
(\forall s \in \RR)\quad 
        \np{\lambda \varphi}^{*}\np{s}
        =
        \lambda
        \varphi^{*}\vardelim{\frac{s}{\lambda}}
        \eqfinp
    \end{equation}

\afterpage{
\begin{landscape}
    \begin{table}
        \centering
        \begin{tabular}{|c|c|c|c|c|c|}
            \hline
            Divergence
            &$\varphi\np{t}$
            &$\varphi\np{t}, t \geq 0$
            &$\phidiv{\prbt}{\probelement}$
            &$\varphi^{*}\np{s}$
            &$\tilde{\varphi}\np{t}$
            \\
            \hline
            \hline
            Kullback-Leibler
            &$\varphi_{kl}\np{t}$
            &$t\log\np{t} -t +1$
            &$\sum_{i = 1}^{N}p_{i}\log\vardelim{\frac{p_{i}}{q_{i}}}$
            &$e^{s}-1$
            &$\varphi_{b}\np{t}$
            \\
            \hline
            Burg entropy
            &$\varphi_{b}\np{t}$
            &$-\log\np{t}+t-1$
            &$\sum_{i = 1}^{N}q_{i}\log\vardelim{\frac{q_{i}}{p_{i}}}$
            &$-\log\np{1-s}\eqsepv s < 1$
            &$\varphi_{kl}\np{t}$
            \\
            \hline
            J-divergence
            &$\varphi_{j}\np{t}$
            &$\np{t-1}\log\np{t}$
            &$\sum_{i=1}^{N}\np{p_{i}-q_{i}}\log\vardelim{\frac{p_{i}}{q_{i}}}$
            &no closed form
            &$\varphi_{j}\np{t}$
            \\
            \hline
            $\chi^{2}$-distance
            &$\varphi_{c}\np{t}$
            &$\frac{1}{t}\np{t-1}^{2}$
            &$\sum_{i=1}^{N}\frac{p_{i}-q_{i}}{p_{i}}$
            &$2-2\sqrt{1-s}, s<1$
            &$\varphi_{mc}\np{t}$
            \\
            \hline
            Modified $\chi^{2}$-distance
            &$\varphi_{mc}\np{t}$
            &$\np{t-1}^{2}$
            &$\sum_{i=1}^{N}\frac{q_{i}-p_{i}}{q_{i}}$
            &$
            \left \{
            \begin{array}{ll}
                -1, &s<-2
                \\
                s+s^{2}/4, &s\geq-2
            \end{array}
            \right .
            $
            &$\varphi_{c}\np{t}$
            \\
            \hline
            Hellinger distance
            &$\varphi_{h}\np{t}$
            &$\bp{\sqrt{t}-1}^{2}$
            &$\sum_{i=1}^{N}\bp{\sqrt{p_{i}}-\sqrt{q_{i}}}$
            &$\frac{s}{1-s},s<1$
            &$\varphi_{h}\np{t}$
            \\
            \hline
            $\chi$-divergence of order $\theta$>1
            &$\varphi_{ca}^{\theta}\np{t}$
            &$\abs{t-1}^{\theta}$
            &$\sum_{i=1}^{N}q_{i}\abs{1-\frac{p_{i}}{q_{i}}}^{\theta}$
            &$s+\np{\theta-1}\vardelim{\frac{\abs{s}}{\theta}}^{\frac{\theta}{\theta-1}}$
            &$t^{1-\theta}\varphi_{ca}^{\theta}\np{t}$
            \\
            \hline
            Variation distance
            &$\varphi_{v}\np{t}$
            &$\abs{t-1}$
            &$\sum_{i=1}^{N}\abs{p_{i}-q_{i}}$
            &$
            \left \{
            \begin{array}{ll}
                -1, &s\leq-1
                \\
                s, &-1 \leq s \leq 1
            \end{array}
            \right .
            $
            &$\varphi_{v}\np{t}$
            \\
            \hline
            Cressie and Read
            &$\varphi_{cr}^{\theta}\np{t}$
            &$\frac{1-\theta+\theta t-t^{\theta}}{\theta\np{1-\theta}}\eqsepv \theta \notin \na{0,1}$
%             \footnote{Note that $\varphi_{cr}^{1}\np{t} = \varphi_{b}\np{t}$ and $\varphi_{cr}^{0}\np{t} = \varphi_{kl}\np{t}$}
            &$\frac{1}{\theta\np{1-\theta}}\vardelim{1-\sum_{i=1}^{N}p_{i}^{\theta}q_{i}^{1-\theta}}$
            &$
            \left \{
            \begin{array}{l}
                \frac{1}{\theta}\bp{1-s\np{1-\theta}}^{\frac{\theta}{\theta-1}}-\frac{1}{\theta}
                \\
                s < \frac{1}{\theta-1}
            \end{array}
            \right .
            $
            &$\varphi_{cr}^{1-\theta}\np{t}$
            \\
            \hline
            Average Value at Risk of level $\beta$
            &$\varphi_{\textrm{avar}}^{\beta}\np{t}$
            &$\iota_{\nc{0, \frac{1}{1-\beta}}}\eqsepv \beta \in [0,1]$
            &$\sum_{i=1}^{N}\iota_{\nc{0, \frac{1}{1-\beta}}}(\frac{p_{i}}{q_{i}})$
            &$\sigma_{\nc{0, \frac{1}{1-\beta}}} =
            \left \{
            \begin{array}{l}
                \frac{1}{1-\beta} \eqsepv s\geq 0
                \\
                0 \eqsepv s < 0
            \end{array}
            \right .
            $
            &$\iota_{[1-\beta,+\infty[}$
            \\
            \hline
        \end{tabular}
        \caption{
            Common perspective functions and their conjugate used to define $\varphi$-divergences.
            %$D_{\varphi}$.
            %The function $\varphi$ is a convex lower semicontinuous function that satisfies
            %$\varphi\np{1} = 0$, $\lim_{x \to \infty} \frac{\varphi\np{x}}{x} = + \infty$,
            %$0\varphi\vardelim{\frac{a}{0}} = \iota_{\RR_{-}}\np{a}$ and that is infinite valued for $t < 0$.
            %The adjoint $\tilde{\varphi}$ is defined for $t \geq 0$ by $\tilde{\varphi} = t\varphi\vardelim{\frac{1}{t}}$
            %and the $\varphi$-divergence by
            %$\phidiv{\prbt}{\probelement} = \sum_{i = 1}^{N} \prbt_{i} \varphi\vardelim{\frac{\probelement_{i}}{\prbt_{i}}}$.
        }
        \label{table_phi_divergence}
    \end{table}
\end{landscape}
}

\subsection{Divergence penalty functions}
% From the definition of a $\varphi$-divergence,
% we derived the definition of a divergence risk measure.
% \begin{mydef}
%     Let $\varphi: \RR \to \nc{0, +\infty}$ be a a lower semi-continuous
%     proper convex function such that $\varphi\np{1} = 0$
%     and let $\alpha: \cM_{1} \to \RR$ be the function given by
%     \begin{equation}
%         \alpha\np{\probelement}
%         =
%         \phidiv{\prbt}{\probelement}
%         \eqfinp
%     \end{equation}
%
%     Then, the \emph{divergence risk measure} $\mesr_{\varphi}$
%     is defined by
%     \begin{align}\label{eq_def_divergence_RM}
%         \nmes{\varphi}{\va{X}}
%         &=
%         \sup_{\probelement \in \cM_{1}}
%         \bp{
%             \nesp{\probelement}{\va{X}}
%             -
%             \alpha\np{\probelement}
%         }
%         \eqfinp
%     \end{align}
% \end{mydef}

A first case of interest is when the penalty term $\alpha(\probelement)$ in Problem \ref{initial_problem_generic}
measures the ``distance'' between  $\prbt$ and $\probelement$ in the sense of a $\varphi$-divergence.
\begin{prop}\label{prop_div}
    Let $\varphi: \RR \to \nc{0, +\infty}$ be a lsc
    convex function with nonempty domain included in $[0,+\infty[$, which is such that $\varphi\np{1} = 0$.
    Let $\alpha$ be the function given by
    \begin{equation}
    (\forall \probelement\in \RR^{N})\quad
        \alpha\np{\probelement}
        = \begin{cases}
        \lambda_{0}\phidiv{\probelement}{\prbt} & \mbox{if $\probelement \in \cM_{1}$}\eqfinv\\
        +\infty & \mbox{otherwise}\eqfinv
        \end{cases}
    \end{equation}
    with $\lambda_{0} \in ]0,+\infty[$.
  Problem \ref{initial_problem_generic} is equivalent to find
    \begin{equation}\label{eq_problem_div}
    \overline{\theta} = \argmin_{\theta\in \RR^{n}}\, \min_{\mu\in \RR} g(\theta,\mu)         
    \eqfinv
    \end{equation}
      where $g$ is the proper, lsc, convex function given by
    \begin{equation}\label{eq_def_g_div}
    (\forall \theta\in \RR^{n})(\forall \mu \in \RR)\quad 
    g(\theta,\mu) =
        \mu
        +
        \sum_{i=1}^{N}
        \prbt_{i}
        \varphi^{*}\Big(\frac{\ell(\theta, z_{i})}{\lambda_{0}}-\mu\Big)
        \eqfinp
    \end{equation}

%    Then, the mapping $\mesr_{\alpha}$ in~\eqref{eq_repres_acceptance}
%    admits the following form
%    \begin{equation}\label{eq_repres_div}
%        \nmes{\alpha}{\va{X}}
%        =
%        \inf_{\mu \in \RR}
%        \mu
%        +
%        \sum_{i=1}^{N}
%        \prbt_{i}
%        \varphi^{*}\bp{x_{i}-\mu}
%        \eqfinp
%    \end{equation}
%    and reformulation~\eqref{eq_repres_acceptance} of problem~\eqref{initial_problem_generic}
%    now reads
%    \begin{equation}\label{eq_problem_div}
%        \inf_{\theta \in \RR^{d}, \mu \in \RR}
%        \mu
%        +
%        \sum_{i=1}^{N}
%        \prbt_{i}
%        \varphi^{*}\bp{\ell \np{\theta, z_{i}}-\mu}
%        \eqfinp
%    \end{equation}

\end{prop}

\begin{proof}
We can reexpress \eqref{eq_initial_problem_generic} as
\begin{equation}\label{eq_initial_problem_generic_renorm} 
    \overline{\theta} \in \argmin_{\theta \in \RR^{n}}
    \sup_{\probelement= (\probelement_{i})_{1\le i \le N} \in \cM_{1}}
    \vardelim{
        \sum_{i = 1}^{N}
        q_{i} 
        \frac{\ell(\theta, z_{i})}{\lambda_{0}}
        -\phidiv{\prbt}{\probelement}
    }
    \eqfinp
%      \tag{GRP}
 \end{equation}
 It follows from \citet*[Theorem 4.122]{follmer2016stochastic} that
     \begin{equation}\label{eq_repres_div}
     (\forall \va{X} \in \XX)\quad
        \nmes{\frac{\alpha}{\lambda_{0}}}{\va{X}}
        =
        \min_{\mu \in \RR}
        \mu
        +
        \sum_{i=1}^{N}
        \prbt_{i}
        \varphi^{*}\bp{x_{i}-\mu}
        \eqfinp
    \end{equation}
    The equivalence between \eqref{eq_initial_problem_generic_renorm} and  \eqref{eq_problem_div} then results from Theorem~\ref{theo_representation}.

    In addition, plugging the expression of $\varphi^*$ in \eqref{eq_def_g_div} yields 
    \begin{equation}\label{eq_def_g_divbis}
    (\forall \theta\in \RR^{n})(\forall \mu \in \RR)\quad 
    g(\theta,\mu) =
        \mu
        +
        \sup_{(t_i)_{i\in \nce{1,N}}\in \RR^{N}} 
        \sum_{i=1}^{N}
        \prbt_{i} t_{i} \Big(\frac{\ell(\theta, z_{i})}{\lambda_{0}}-\mu\Big)- \varphi(t_{i})
        \eqfinp
    \end{equation}
    For every $(t_i)_{i\in \nce{1,N}}\in \RR^{N}$,
    \begin{equation}
    (\theta,\mu) \mapsto  \sum_{i=1}^{N}
        \prbt_{i} t_{i} \Big(\frac{\ell(\theta, z_{i})}{\lambda_{0}}-\mu\Big)- \varphi(t_{i})
    \end{equation}
    is a lsc convex function. Since convexity and lower semicontinuity are
    kept by the supremum operation, $g$ is lsc and convex.
   By using \eqref{eq_loss_prop}, \eqref{eq_def_g_div}, and the fact that $\varphi^*$ is proper, there exist $\theta\in \RR^{n}$
   and $\mu \in \RR$, such that $g(\theta,\mu)<+\infty$.

\end{proof}

\subsection{Constrained formulations}
We now investigate two particular cases
when $\alpha$ is the indicator function of a convex set $\probset$ of probability distributions, so defining an ambiguity set.

\subsubsection{Ball with respect to a divergence}
A first possibility is to introduce an upper bound on the divergence $D_{\varphi}\np{\probelement,\prbt}$ 
between the sought
distribution $\probelement$ and $\prbt$ by considering the constraint set
\begin{equation}
    \probset=\BB_{\epsilon}^{\varphi}
    =
    \Ba{q \in \cM_{1} \mid D_{\varphi}\np{\probelement,\prbt} \leq \epsilon}
    \eqfinv
\end{equation}
where $\epsilon \in ]0,+\infty[$.

%Here is our first result where we provide a convex reformulation
%assuming that the function $\alpha$ is an indicator function.
%This example comes from existing literature.
%Let $\varphi$ be function as in Definition~\ref{def_phi_divergence}
%and let $\epsilon > 0$.
%We define the set
%\begin{equation}
%    \BB_{\epsilon}^{\varphi}
%    =
%    \Ba{q \in \cM_{1} | D_{\varphi}\np{\prbt,\probelement} \leq \epsilon}
%    \eqfinv
%\end{equation}
%as a set of probability distribution ``not far''
%from the probability distribution $\prbt$.

The following result generalizes both \citet{ben2013robust}
where the authors deal with linear costs under constraints and \citet{hu2013kullback} where the authors 
focus on the Kullback-Leibler divergence.

% \com{Check that is $\lambda \phi$ is a divergence}
\begin{prop}\label{prop_reformulation_RM}
   Let $\varphi: \RR \to \nc{0, +\infty}$ be a lsc
    convex function such that $\dom(\varphi) = ]0,+\infty[$ or $\dom(\varphi) = [0,+\infty[$, and
    $\varphi\np{1} = 0$.
%    Assume that there exists $\eta \in ]0,+\infty[$ such that $]0,1+\eta[ \in \dom(\varphi)$.
%     such that $1$ is a interior point of the domain of $\varphi$ .    
    Let $\epsilon \in ]0,+\infty[$ and let $\alpha= \iota_{\BB_{\epsilon}^{\varphi}}$.
%    Then Problem \ref{initial_problem_generic} can be expressed as
%    \begin{equation}
%       \mbox{find} \quad \overline{\theta}\in \argmin g
%    \end{equation} 
%    where $g$ is the proper, lower-semincontinuous, convex function defined as
%     \begin{align}\label{eq_align_reformulation_phi}
%     (\forall \theta \in \RR^{d})\quad
%        g(\theta)=\inf_{\lambda \geq 0, \mu\in \RR}
%        \lambda \epsilon
%        + \mu
%        +\sum_{i = 1}^{N}
%        p_{i}
%            \lambda
%            \varphi^{*}
%            \vardelim{
%                \frac{\ell \np{\theta, z_{i}} -\mu}{\lambda}
%        }
%        \eqfinp
%    \end{align}
%   
%   
Problem \ref{initial_problem_generic} is equivalent to find
    \begin{equation}\label{eq_problem_div_ball}
    \overline{\theta} = \argmin_{\theta\in \RR^{n}} \min_{(\lambda,\mu)\in \RR^{2}} g(\theta,\lambda,\mu)         
    \eqfinv
    \end{equation}
    where $g$ is the proper, lsc, convex function given by
    \begin{multline}\label{eq_def_g_div_ball}
    (\forall \theta\in \RR^{n})(\forall (\lambda,\mu) \in \RR^{2})\
    g(\theta,\lambda,\mu) =\\
    \begin{cases}
        \displaystyle \lambda \epsilon
        + \mu
        +\sum_{i = 1}^{N}
        p_{i}
            \lambda
            \varphi^{*}
            \vardelim{
                \frac{\ell \np{\theta, z_{i}} -\mu}{\lambda}      
        } & \mbox{if $\lambda \in [0,+\infty[$} \eqfinv\\
        +\infty & \mbox{otherwise}        \eqfinv
        \end{cases}
    \end{multline}
    with the convention 
    \begin{equation}\label{eq_conv_phistar0}
    0 \varphi^{*}\left(\frac{\cdot}{0}\right) = \iota_{]-\infty,0]}
        \eqfinp    
    \end{equation}

%        Then, the coherent risk measure $\mesr_{\alpha}$ in~\eqref{eq_repres_acceptance}
%    admits the following form
%    \begin{equation}\label{eq_repres_divset}
%        \nmes{\alpha}{\va{X}}
%        =
%        \inf_{\lambda \geq 0, \mu}
%        \lambda \epsilon
%        + \mu
%        +\sum_{i = 1}^{N}
%        p_{i}
%            \lambda
%            \varphi^{*}
%            \vardelim{
%                \frac{x_{i} -\mu}{\lambda}
%        }
%        \eqfinv
%    \end{equation}
%%    and reformulation~\eqref{eq_repres_acceptance} of problem~\eqref{initial_problem_generic}
%    now reads
%    \begin{align}\label{eq_align_reformulation_phi}
%        \inf_{\theta, \lambda \geq 0, \mu}
%        \lambda \epsilon
%        + \mu
%        +\sum_{i = 1}^{N}
%        p_{i}
%            \lambda
%            \varphi^{*}
%            \vardelim{
%                \frac{\ell \np{\theta, z_{i}} -\mu}{\lambda}
%        }
%        \eqfinp
%    \end{align}
\end{prop}

%We now entering the proof of Proposition~\ref{prop_reformulation_RM}.
%We have inspired ourselves from~\citet{ben2013robust} where the authors
%deals with constraints and linear costs. The proof is a generalization
%of the one found in~\citet{hu2013kullback} where the authors have focused on the Kullback Leibler
%divergence.

% \com{The case $\alpha\np{\probelement} = D_{\varphi}\np{\prbt,\probelement}$
%     can be retrieve as a limit case when $\epsilon \to \infty$.}
% \com{Clarifier le message avec mesure entropique et ambiguity set entropique}
% Then Equation~\eqref{eq_align_reformulation_phi} reads exactly as in
% Theorem~\ref{theo_rep_divergence_RM}.

\begin{proof}
%    We consider the Problem~\ref{initial_problem_generic}
The risk function associated with $\alpha = \iota_{\BB_{\epsilon}^{\varphi}}$ is 
%    for every $\theta\in \RR^{d}$,
    \begin{subequations}
        \begin{align}
        (\forall \va{X}\in \XX) \quad
%            h(\theta)
            \nmes{\alpha}{\va{X}}=\sup_{q \in \RR^{N}}
            &\quad
            \sum_{i = 1}^{N}
            q_{i} x_{i} - \iota_{\cM_{1}}(q)
            %\ell \np{\theta, z_{i}}
            \eqfinv
            \\
            \textrm{s.t.}
            &
            \quad
            \sum_{i=1}^{N}
            p_{i}\varphi\vardelim{\frac{q_{i}}{p_{i}}}
            \leq \epsilon \label{e:constdivball}
            \eqfinp
        \end{align}
    \end{subequations}
    Since $1$ belongs to the interior of $\dom(\varphi)$ and
    \begin{equation}
     \sum_{i=1}^{N}
            p_{i}\varphi\vardelim{\frac{p_{i}}{p_{i}}} = 0 <
            \epsilon
            \eqfinv
            \end{equation}
	Slater's condition holds for constraint \eqref{e:constdivball}. Since the constraint is feasible and $q\mapsto -\sum_{i = 1}^{N} q_{i} x_{i} + \iota_{\cM_{1}}(q)$ is lsc, convex, and coercive, there exists a solution $\overline{q} \in \cM_{1}$ to the above constrained maximization problem. It then follows from standard Lagrange duality for convex functions that
	there exists $\overline{\lambda} \in [0,+\infty[$ such that $(\overline{q},\overline{\lambda})$ is a saddle point of the Lagrange function
	\begin{align}
	(\forall q \in \mathcal{C})(\forall \lambda \in [0,+\infty[)\;\; \Psi_{\va{X}}(q,\lambda) = -\sum_{i = 1}^{N} q_{i} x_{i} + \lambda\left(\sum_{i=1}^{N}
            p_{i}\varphi\vardelim{\frac{q_{i}}{p_{i}}}-\epsilon\right) 
             \eqfinv
           \end{align}
           where
               \begin{equation}
    \mathcal{C} = \left\{q \in  \cM_{1} \mid (\forall i \in \nce{1,N})\; \frac{q_{i}}{p_{i}}\in \dom \varphi\right\} \eqfinp
    \end{equation}

           We have thus 
             \begin{equation}\label{eq_mesriskFG}
            \nmes{\alpha}{\va{X}}= -\sup_{\lambda \in [0,+\infty[}\inf_{q \in \mathcal{C}} 
             \Psi_{\va{X}}(q,\lambda)       
            = \min_{\lambda \in [0,+\infty[} G(\va{X},\lambda) = G(\va{X},\overline{\lambda})
            \eqfinv
            \end{equation}
            where, for every $\lambda \in [0,+\infty[$,
            \begin{equation}\label{eq_prob_inter_proof}
            G(\va{X},\lambda) =
            \lambda \epsilon
            +
            \sup_{q \in \mathcal{C}}
            \sum_{i = 1}^{N}
            \Bgp{
            q_{i} x_{i}
            %\ell \np{\theta, z_{i}}
            -\lambda p_{i}
            \varphi\vardelim{\frac{q_{i}}{p_{i}}}
            }
            \eqfinp
         \end{equation}
         Two cases will be distinguished.
\begin{enumerate}
   \item Case when $\lambda = 0$.\\
   Then \eqref{eq_prob_inter_proof} reduces to
   \begin{equation} \label{e:supportfM1prev}
        G(\va{X},0)
        = \sup_{q \in \mathcal{C}}
            \sum_{i = 1}^{N}
            q_{i} x_{i} \le  \sigma_{\cM_{1}}(\va{X})
         \eqfinv
    \end{equation}
    where
       \begin{equation} \label{e:supportfM1}
         \sigma_{\cM_{1}}(\va{X}) =\sup_{q \in \cM_{1}}
        \sum_{i = 1}^{N}
        q_{i} x_{i} =
        \sup_{i \in \nce{1,N}}
        x_i\eqfinp
       \end{equation}          
       In addition, since $]0,+\infty[ \in \dom(\varphi)$, the upper bound in \eqref{e:supportfM1prev} is attained, yielding
    \begin{equation} \label{e:supportfM1post}
        G(\va{X},0)
 =  \sup_{i \in \nce{1,N}}
        x_i\eqfinp    
        \end{equation}

%    and Problem \eqref{initial_problem_generic} is equivalent to find
%    \begin{align}
%    \overline{\theta} \in \argmin_{\theta\in \RR^{d}}
%         \sup_{i \in \nce{1,n}}
%        \ell \np{\theta, z_{i}}
%        \eqfinp
%    \end{align}
   \item Case when $\lambda > 0$.\\
   \eqref{eq_prob_inter_proof} can be reexpressed as
   \begin{align}
             G(\va{X},\lambda) &=  \lambda \epsilon
            +
            \sup_{q \in \cM_{1}}
            \sum_{i = 1}^{N}
            \Bgp{
            q_{i} x_{i}
            %\ell \np{\theta, z_{i}}
            -\lambda p_{i}
            \varphi\vardelim{\frac{q_{i}}{p_{i}}}
            }\nonumber\\
             &= \lambda \epsilon + (\lambda \Phi+\iota_{\cM_{1}})^*(\va{X})
              \eqfinv
              \end{align}
              where
              \begin{equation}
             (\forall q \in \RR^N) \quad \Phi(q) = \sum_{i=1}^N p_{i} \varphi\vardelim{\frac{q_{i}}{p_{i}}}
             \eqfinp
             \end{equation}
             The conjugate of $\Phi$ reads 
             \begin{align}\label{eq_conj_Phi}
             (\forall \va{Y} \in \XX) \quad (\lambda\Phi)^{*}(\va{Y}) &=  \sup_{q\in \RR^{N}} q_{i} y_{i}-\lambda p_{i} \varphi\vardelim{\frac{q_{i}}{p_{i}}}\nonumber\\
             & = \sum_{i=1}^{N} p_{i} (\lambda \varphi)^{*}(y_{i})
             \eqfinv
             \end{align}
             whereas the conjugate of $\iota_{\cM_{1}}$ is given by $\sigma_{\cM_{1}}$ in \eqref{e:supportfM1}.
              Since $\sigma_{\cM_{1}}$ is finite valued, the conjugate of $\Phi+\iota_{\cM_{1}}$ is given by the following inf-convolution \cite[Theorem 15.3]{bouquin_combettes}
             \begin{equation}\label{eq_inf_conv}
             (\Phi+\iota_{\cM_{1}})^{*}(\va{X}) = \min_{\va{Y}\in \XX} \sigma_{\cM_{1}}(\va{Y})+(\lambda\Phi)^*(\va{X}-\va{Y})
             \eqfinv
             \end{equation}
             which, by using \eqref{eq_conj_Phi}, yields
             \begin{equation}
             G(\va{X},\lambda) = \lambda \epsilon +\min_{\substack{\va{Y}\in \XX\\ \sup_{i \in \nce{1,N}}y_{i} = \mu}}  \mu+\sum_{i=1}^{N} p_{i} (\lambda \varphi)^{*}(x_{i}-y_{i})
              \eqfinp
 \end{equation}
 Since $\dom(\lambda \varphi) \subset [0,+\infty[$, $(\lambda \varphi)^{*}\colon \xi\mapsto \sup_{\upsilon \in [0,+\infty[} \xi \upsilon -\lambda \varphi(\upsilon)$ is an increasing function. This implies that
             \begin{align}
             G(\va{X},\lambda) &= \lambda \epsilon +\min_{\mu \in \RR}  \mu+\sum_{i=1}^{N} p_{i} (\lambda \varphi)^{*}(x_{i}-\mu)\nonumber\\
             &= \lambda \epsilon +\min_{\mu \in \RR}  \mu+\sum_{i=1}^{N} p_{i} \lambda \varphi^{*}\left(\frac{x_{i}-\mu}{\lambda}\right)
              \eqfinp
 \end{align}
 \end{enumerate}
 Note that the right-hand side in the previous formula when applied at $\lambda = 0$ by using \eqref{eq_conv_phistar0} and
\eqref{e:supportfM1} reduces to
\begin{align}
             &\min_{\mu \in \RR}  \mu+\sum_{i=1}^{N} p_{i} \,\iota_{]-\infty,0]}(x_{i}-\mu)\nonumber\\
             =&\; \min_{\substack{\mu \in \RR\\ (\forall i \in \nce{1,N})\, x_{i}\le \mu}}  \mu\nonumber\\
             = &\; G(\va{X},0)
             \eqfinp
\end{align}
Consequently, \eqref{eq_mesriskFG} leads to
        \begin{equation}\label{eq_repres_divset}
             \nmes{\alpha}{\va{X}}= \min_{\lambda \in [0,+\infty[,\mu \in \RR}  
             \lambda \epsilon + \mu+\sum_{i=1}^{N} p_{i} \lambda \varphi^{*}\left(\frac{x_{i}-\mu}{\lambda}\right)
             \eqfinv
             \end{equation}
and \eqref{eq_problem_div_ball} follows from Theorem \ref{theo_representation}.

In addition, by using the expression of the conjuguate, $g$ can be reexpressed as
    \begin{multline}
    (\forall \theta\in \RR^{n})(\forall (\lambda,\mu) \in \RR^{2})\\
    g(\theta,\lambda,\mu)= \sup_{(t_i)_{i\in \nce{1,N}}\in \RR^{N}} 
         \lambda \epsilon
        + \mu
        +\sum_{i = 1}^{N}
        p_{i} (\ell \np{\theta, z_{i}} -\mu)t_{i}-\lambda \varphi(t_{i})) + \iota_{[0,+\infty[}(\lambda).
 \end{multline}
 As a supremum of lsc convex functions, $g$ also is lsc convex.
 The fact that $g$ is proper follows from arguments similar  to those at the end of the proof of Proposition \ref{prop_div}.

\end{proof}

% \subsubsection{Examples with Entropic Risk Measure and Average Value at Risk}
%
% \begin{cor}
%     Let $\varphi_{kl} = t\log\np{t} -t +1$
%     be the function introduced in Definition~\ref{def_entropic_RM}
%     and assume that $\alpha$ is the indicator function
%     of the set $\iota_{\BB_{\epsilon}^{\varphi_{kl}}}$.
%     The Problem~\ref{initial_problem_generic} is equivalent to
%     \begin{align}
%         \inf_{\theta, \lambda \geq 0, \mu}
%         \lambda \np{\epsilon - 1}
%         + \mu
%         + \sum_{i = 1}^{N}
%             p_{i}
%             \lambda
%             \exp
%             \vardelim{
%                 \frac{\ell \np{\theta, z_{i}} -\mu}{\lambda}
%         }
%         \eqfinp
%     \end{align}
% \end{cor}
%
%
% \begin{cor}
%     Let $\varphi_{avar}^{\beta} = \iota_{\nc{0, \frac{1}{1-\beta}}}$
%     be the function introduced in Definition~\ref{def_avar_RM}
%     and assume that $\alpha$ is the indicator function
%     of the set $\iota_{\BB_{\epsilon}^{\varphi_{avar}^{\beta}}}$.
%     The Problem~\ref{initial_problem_generic} is equivalent to
%     \begin{subequations}
%     \begin{align}
%         \inf_{\theta, \lambda \geq 0, \mu}
%         &\quad
%         \lambda \epsilon
%         + \mu
%         + \frac{1}{1-\beta}
%         \sum_{i = 1}^{N}
%         p_{i}s_{i}
%         \\
%         \textrm{s.t.}
%         &\quad
%         \ell \np{\theta, z_{i}} -\mu \leq s_{i}
%         \\
%         &\quad
%         0\leq s_{i}
%         \eqfinp
%     \end{align}
%     \end{subequations}
% \end{cor}

\begin{rem}
The divergence risk measure
in~\eqref{eq_repres_div} is convex, whereas 
the risk measure  in~\eqref{eq_repres_divset} is coherent (see Proposition~\ref{prop:preconvex}), which means that
the risk scales with the data in the latter case.
\end{rem}

\subsubsection{Ball with respect to the Wasserstein metric}\label{section_DRO}
We now investigate Problem~\ref{initial_problem_generic}
when function $\alpha$ is the indicator of a Wasserstein ball centered on $\prbt$.
For this purpose, we first recall the notion of Wasserstein distance.

%For the recall, we give the definition of the Wasserstein distance
\begin{mydef}
    Let $\mathcal{M}\np{\Xi^{2}}$ denote the set of probability distributions supported on $\Xi^2$.
    The Wasserstein distance between two distributions $\prbt$ and $\probelement$ supported on $\Xi$ is defined as
    \begin{equation}
        W\np{p, q}
        =
        \inf_{\Pi \in M\np{\Xi^{2}}}
        \Bga{
            \int_{\Xi^{2}}
            \delta\np{\xi, \xi'}
            \Pi\np{d\xi, d\xi'} \mid
            \Pi\np{d\xi, \Xi}
            =
            q\np{d\xi},
            \Pi\np{\Xi, d\xi'}
            =
            p\np{d\xi'}
        }
        \eqfinv
    \end{equation}
    where %$\xi = \np{x, y}$ and 
    $\delta$ is a metric on $\Xi$.
\end{mydef}

We now introduce the notion of Wasserstein ball.
The considered constrained set is denoted by 
\begin{equation}
    \probset=\BB_{\epsilon}^{\WW}
    =
    \Ba{q \in \cM_{1} \mid W\np{\prbt,\probelement} \leq \epsilon}
\end{equation}
with $\epsilon \in ]0,+\infty[$.

%$\probset=\BB_{\epsilon}^{\WW}\np{p}$ the ball of radius 
%centered on $p$ with respect to the Wasserstein distance.
In the following theorem, $\delta$ is the usual Euclidean distance.
The following convex reformulation of Problem \ref{initial_problem_generic} can
be derived from \cite[Theorem 4.2]{esfahani2015data}.
%This problem thus formulated is a saddle point problem
%which is non convex and we recall a convex reformulation provided
%in~\citet*{esfahani2015data}.

\begin{prop}\label{prop_reformulation_kuhn}
%    Let $\alpha: \cM_{1} \to \RR$ be the function given by
%    \begin{equation}
%        \alpha\np{\probelement}
%        =
%        \iota_{\BB_{\epsilon}^{\WW}\np{\prbt}}
%        \eqfinp
%    \end{equation}
    Let $\epsilon \in ]0,+\infty[$ and let $\alpha = \iota_{\BB_{\epsilon}^{\WW}}$.
    Then,  Problem \ref{initial_problem_generic} is equivalent to find
    \begin{equation}\label{eq_reformulation_wasser}
    \overline{\theta} = \argmin_{\theta\in \RR^{n}} \min_{\lambda\in \RR,s\in \RR^N} g(\theta,\lambda,s)         
    \eqfinv
    \end{equation}
      where $g$ is the proper, lsc convex function given by
    \begin{equation}
    (\forall \theta\in \RR^{n})(\forall \lambda \in \RR)(\forall s = (s_{j})_{1\le j \le N} \in \RR^N)\quad 
    g(\theta,\lambda,s) =
        \lambda \epsilon
            +
            \sum_{j = 1}^{N}
            p_{j}s_{j} +\iota_{\mathcal{W}}(\theta,\lambda,s)    
            \eqfinv
    \end{equation}
    where $\mathcal{W}$ is the closed convex set defined as
    \begin{equation}
    \mathcal{W} = \{(\theta,\lambda,s)\in \RR^n\times [0,+\infty[ \times \RR^{N} \mid (\forall (i,j) \in \nce{1,N}^2)\;
    \ell \np{\theta, z_{i}} -\lambda \nnorm{z_{i} - z_{j}}
            \leq s_{j}            \}
            \eqfinp
    \end{equation}

%    Then, the mapping $\mesr_{\alpha}$ in~\eqref{eq_repres_acceptance}
%    admits the following form
%    \begin{subequations}\label{eq_repres_was}
%        \begin{align}
%            \nmes{\alpha}{\va{X}}
%            =
%            \inf_{\lambda \geq 0, s_{j}}
%            & \quad
%            \lambda \epsilon
%            +
%            \sum_{j = 1}^{N}
%            p_{j}s_{j}
%            \\
%            \textrm{s.t.} &\quad
%            x_{i} -\lambda \nnorm{x_{i} - x_{j}}
%            \leq s_{j}
%            \eqfinv
%            \\
%            &\quad
%            \forall i \in \nce{1;N}
%            \eqsepv \forall j \in \nce{1;N}
%            \eqfinp
%            \nonumber
%        \end{align}
%    \end{subequations}
%    and reformulation~\eqref{eq_repres_acceptance} of problem~\eqref{initial_problem_generic}
%    now reads
%    \begin{subequations}\label{eq_reformulation_wasser}
%        \begin{align}
%            \inf_{\theta, \lambda \geq 0, s_{j}}
%            & \quad
%            \lambda \epsilon
%            +
%            \sum_{j = 1}^{N}
%            p_{j}s_{j}
%            \\
%            \textrm{s.t.} &\quad
%            \ell \np{\theta, z_{i}} -\lambda \nnorm{z_{i} - z_{j}}
%            \leq s_{j}
%            \eqsepv \forall i \in \nce{1;N}
%            \eqsepv \forall j \in \nce{1;N}
%            \eqfinp
%        \end{align}
%        \end{subequations}
\end{prop}

\section{Numerical solution}\label{section_numerical_experiments}
%In Sect.~\ref{section_risk_measures}, we have studied risk measure
%and provided a convex reformulation to the saddle point Problem~\ref{initial_problem_generic}.
%Then in Sect.~\ref{section_literature}, we have revisited examples of the literature.
We will now propose an algorithm allowing us to solve numerically the three convex optimization problems in Propositions
\ref{prop_div}, \ref{prop_reformulation_RM}, and \ref{prop_reformulation_kuhn}. This algorithm 
applies to more general choices of function $\alpha$ in Problem~\ref{initial_problem_generic}
where the constraint $\mathcal{S}$ in \eqref{e:defS} splits
as an intersection of a finite number of convex constraints.
%However, for clarity, we prefer to focus on the tree examples provided in Sect.~\ref{section_literature}.

\subsection{A unifying formulation}
We first show that the convex optimization problems discussed in Section \ref{section_literature} can be reexpressed in a unifying manner.
\begin{prop}
The optimization problems in Propositions 
\ref{prop_div}, \ref{prop_reformulation_RM}, and \ref{prop_reformulation_kuhn}
amount to finding
%can all be expressed as finding
\begin{subequations}\label{eq_common_formulation}
    \begin{align}
       (\overline{\theta},\overline{\lambda},\overline{\mu},\overline{s}) \in \argmin_{\theta\in \RR^{n},\lambda \in [0,+\infty[, \mu\in \RR, s \in \RR^{N}}
        &\quad
        \lambda \epsilon
        +\mu
        + \sum_{i = 1}^{N}
        p_{i}
        s_{i}
        \\
        \textrm{s.t.}
        &\quad (\forall k \in \nce{1,K})\quad
        f_{k}\bp{\theta, \lambda, \mu, \mathbf{z}}
        \leq 0
        \eqfinv
    \end{align}
\end{subequations}
where $K\in \NN\setminus\{0\}$ and the functions $\big(f_{k}\bp{\cdot, \mathbf{z}}\big)_{k\in \nce{1,K}}$ are proper, lsc, and convex.
More precisely,
\begin{enumerate}
\item for divergence penalty functions, $K=N$ and, for every $k \in \nce{1,K}$,
        \begin{multline}\label{e:fkdivpen}
            (\forall \theta\in \RR^{n})(\forall \lambda\in [0,+\infty[)(\forall \mu\in \RR)(\forall s \in \RR^{N})\\
            \quad f_{k}\bp{\theta, \lambda, \mu, s,\mathbf{z}}
            =
            \varphi^{*}\left(\frac{\ell \np{\theta, z_{k}}}{\lambda}-\mu\right)+\iota_{\{\lambda_{0}\}}(\lambda)-s_{k}
            \eqfinv
        \end{multline} 
\item for divergence ball constraints, $K=N$ and, for every $k \in \nce{1,K}$,       
        \begin{multline}
            (\forall \theta\in \RR^{n})(\forall \lambda\in [0,+\infty[)(\forall \mu\in \RR)(\forall s \in \RR^{N})\\
            \quad f_{k}\bp{\theta, \lambda, \mu, s,\mathbf{z}}
            =
            \lambda
            \varphi^{*}
            \vardelim{
                \frac{\ell \np{\theta, z_{k}} -\mu}{\lambda}
                }-s_{k}
                \eqfinv
        \end{multline} 
\item for the Wassertein ball constraint, $K=N^{2}$ and, for every $k \in \nce{1,K}$
and $(i_{k},j_{k})\in   \nce{1,N}^{2}$ such that $k = N(i_{k}-1) + j_{k}$,
        \begin{multline}
            (\forall \theta\in \RR^{n})(\forall \lambda\in [0,+\infty[)(\forall \mu\in \RR)(\forall s \in \RR^{N})\\
            \quad f_{k}\bp{\theta, \lambda, \mu, s, \mathbf{z}}
            =
            \ell \np{\theta, z_{i_{k}}} -\lambda \nnorm{z_{i_{k}} - z_{j_{k}}}-s_{j_{k}}
            \eqfinp
        \end{multline} 
\end{enumerate}
\end{prop}

\subsection{Description of the algorithm}
In this section, we propose  an accelerated projected gradient algorithm for solving
Problem~\eqref{eq_common_formulation}.
One step of this proximal algorithm \cite{articletutoriel-Combettes-Pesquet}  reads as a projection onto a set
defined as an intersection of non trivial closed convex sets.
To solve this projection problem, we use the subgradient projection algorithm
in~\citet*{combettes2003block}, which is related to ideas introduced
in~\citet*[Theorem 3-2]{haugazeau1968inequations}. This algorithm allows 
the constraints to be activated individually in a flexible parallel manner.
%This section is organized as follows: 
We will first recall the basic structure of our algorithm
%basic notions
%about proximal algorithms to illustrate why the projection algorithm
%is a key idea. Then, we develop the 
before describing in more details the subgradient projection step.
%projection algorithm
%proposed in~\citet*{combettes2003block} that
%takes advantage of parallel computation.
%Finally, we applied this algorithm on examples.
% and compare with existing methods.
% \com{Faire le lien avec Ogazo}

\paragraph{Proximal algorithm}
%Let $\cH$ be an Hilbert space and let $f\colon \cH \to \RR \cup \na{+\infty}$ be a proper, lsc, convex.
%We define the proximity operator \cite{articletutoriel-Combettes-Pesquet}  of $f$ at $u\in \cH$ by
%\begin{equation}
%    \operatorname{prox}_{f}\np{u}
%    =
%    \argmin_{v \in \cH}
%    f\np{v} + \frac{1}{2}\nnorm{u-v}^{2}
%    \eqfinp
%\end{equation}
%where $\nnorm{\cdot}$ (resp. $\proscal{\cdot}{\cdot}$) are the norm (resp. the inner product) equiping $\cH$.

Let $\cH = \RR^n\times \RR \times \RR \times \RR^N$
and let $\nnorm{\cdot}$ (resp. $\proscal{\cdot}{\cdot}$) denote the standard norm (resp. the inner product) equiping this product space.
By introducing the generic variable $u=(\theta,\lambda,\mu,s)\in \cH$,
\eqref{eq_common_formulation} can be reexpressed more concisely as
\begin{equation}\label{eq_prob_min_on_C}
    \min_{u \in \cH} \proscal{c}{u}+ \iota_{\mathcal{C}}\np{u}
    \eqfinv
\end{equation}
where $c = (0,\epsilon,1,p)\in \cH$ and $\mathcal{C} = \cap_{k=0}^K \mathcal{C}_{k}$
with
\begin{align}
&\mathcal{C}_{0} = \{ (\theta,\lambda,\mu,s)\in \RR^N\times [0,+\infty[ \times \RR \times \RR^N\}
    \eqfinv\\
&(\forall k \in \nce{1,K})\quad  \mathcal{C}_{k} = \operatorname{lev}_{\le 0} f_{k}\bp{\cdot, \mathbf{z}}
\eqfinp
\end{align}
To solve the above problem, we propose to employ a FISTA-like algorithm \cite{Beck-Teboule}.
Let $n\in \NN\setminus\{0\}$. The $n$-th iteration of this algorithm reads
\begin{align}
v^{(n)}&= u^{(n)}+\frac{\tau^{(n)}-1}{\tau^{(n+1)}} (u^{(n)}-u^{(n-1)}) \eqfinv\\
u^{(n+1)}&= \operatorname{P}_{\mathcal{C}}(v^{(n)}-\gamma c) \eqfinv \label{e:iterFISTA}
\end{align}
where $\gamma \in ]0,+\infty[$ and
$\operatorname{P}_{\mathcal{C}}\colon \cH \to \mathcal{C}$ is the projection onto the closed 
convex set $\mathcal{C}$. 
It follows from \cite[Theorem 3]{Chambolle-Dossal} that, if a solution to the minimization problem exists, and
\begin{equation}
\tau^{(n)} = \frac{n+a-1}{a}, \quad a > 2 \eqfinv
\end{equation}
then the convergence of $(u^{(n)})_{n\in \NN}$ to a solution to the problem is guaranteed.
%Let $n>1$ be an integer and
%assume that you want to solve the following problem
%\begin{equation}\label{eq_prob_min_on_C}
%    \min_{x \in \RR^{n}}
%    c^{T}x + \iota_{C}\np{x}
%    \eqfinv
%\end{equation}
%where $f: \RR^{n} \to \RR$ is a numerical function
%and $C$ is a convex subset of $\RR^{n}$.
%Then, the accelerated proximal step reads
%\begin{subequations}
%    \begin{align}
%        y &= x^{\np{k-1}} + \frac{k-2}{k+1}\np{x^{\np{k-1}} - x^{\np{k-2}}}
%        \eqfinv
%        \\
%        x^{\np{k}} &= P_{C}\np{y - t_{k}c}\eqfinv
%    \end{align}
%\end{subequations}
%where $P_{C} : \RR^{n} \to C$ is the projection mapping onto the convex set $C$.

The main difficulty in the implementation of the algorithm lies in the computation of the projection onto 
$\mathcal{C}$ that will be discussed next.
%We let the reader convinces himself that Problem~\eqref{eq_common_formulation}
%is a particular case of Problem~\eqref{eq_prob_min_on_C}.
%To solve Problem~\eqref{eq_common_formulation} with an accelerated proximal step,
%we need an algorithm to compute a projection onto a convex set that
%is scalable.

\paragraph{Computation of the projection}
Algorithm \ref{algo_projection} presents our  projection method inspired from~\citet*{combettes2003block}. At iteration $\ell \in \NN$, $Q(p^{(0)},p^{(\ell)},r^{(\ell)})$ designates 
the projection of $p^{(0)}$ onto the intersection of the 3 half-spaces
$\mathcal{C}_{0}$, $H_\ell$, and $D_\ell$, where
\begin{align}
H_{\ell} &= \{u \in \cH \mid  \proscal{u-r^{(\ell)}}{p^{(\ell)}-r^{(\ell)}} \le 0 \} \eqfinv\\
D_{\ell} &= \{u \in \cH \mid  \proscal{u-p^{(\ell)}}{p^{(0)}-p^{(\ell)}} \le 0 \} \eqfinp
\end{align}
Since the projection onto $H_{\ell} \cap D_{\ell}$ has an explicit form \citet*{combettes2003block},
a dual forward-backward algorithm \cite{CombettesDFB} allows us to compute in a fast manner the projection onto
$\mathcal{C}_{0} \cap H_{\ell} \cap D_{\ell}$. The algorithm has been intialized by setting $p^{(0)} = \operatorname{P}_{\mathcal{C}_{0}}(v^{(n)}-\gamma c)$, taking into account the fact that $\operatorname{P}_{\mathcal{C}} = \operatorname{P}_{\mathcal{C}} \circ \operatorname{P}_{\mathcal{C}_{0}}$. At each iteration $\ell$, $\mathbb{K}_{\ell}$ designates the set of indices of the constraints which are activated. When dealing with large-scale problems, it may be useful not to require all the constraints to be activated at each iteration.
The convergence of the algorithm is guaranteed by the study in \cite{outerCombettes2000}, provided that, for every $k\in \nce{1,K}$, $\mathcal{C}_{0} \subset \dom(\partial f_{k})$ and there exists an integer $M_{k}$ such that 
\begin{equation}
(\forall \ell \in \NN)\quad  
k \in \bigcup_{s=\ell}^{\ell+M_{k}} \mathbb{K}_{s}.
\end{equation}
The first assumption on the domains of the subdifferentials of the functions $(f_{k})_{k\in \nce{1,K}}$ is however not satisfied in
\eqref{e:fkdivpen}. In this case, the direct simpler form of the algorithm in \citet*{combettes2003block} can be applied
 since the parameter $\lambda$ is fixed.

\IncMargin{1em}
\begin{algorithm}[H]\label{algo_projection}
    \SetAlgoLined
    \KwData{$u^{(n-1)} \in \cH$, $u^{(n)} \in \cH$, $\delta \in ]0,1[$}
    \KwResult{Output of the accelerated projected gradient iteration \eqref{e:iterFISTA}}
    $\displaystyle v^{(n)}= u^{(n)}+\frac{\tau^{(n)}-1}{\tau^{(n+1)}} (u^{(n)}-u^{(n-1)})$\;
    $p^{(0)} = \operatorname{P}_{\mathcal{C}_{0}}(v^{(n)}-\gamma c)$\;
    Initialize $\ell = 0$,
    \While{$p^{(\ell)}\notin \mathcal{C}$}{
%    \While{$\np{y-\gamma_{n}c} = x_{n}  \notin \cap_{i \in I}S_{i}$}{
        Take a nonempty finite index set $\mathbb{K}_{\ell} \subset \nce{1,K}$\;
%         \begin{enumerate}
            For every $k \in \mathbb{K}_{\ell}$,
            $p_{k}^{(\ell)}
            =
            \left \{
            \begin{array}{lr}
                \displaystyle p^{(\ell)} - \frac{f_{k}\np{p^{(\ell)} }t_{k}^{(\ell)}}{\nnorm{t_{k}^{(\ell)}}}
                \eqsepv
                t_{k}^{(\ell)} \in \partial f_{k}\np{p^{(\ell)} }
                \eqfinv
                & \textrm{if } f_{k}\np{p^{(\ell)} } > 0
                \\
                p^{(\ell)}  \eqfinv
                & \textrm{if } f_{k}\np{p^{(\ell)} } \leq 0
            \end{array}
            \right .
            $
            \\
            Choose $\{\omega_{k,\ell}\mid k \in \mathbb{K}_{\ell}\} \subset \nc{\delta,1}$ such that $\displaystyle\sum_{k \in \mathbb{K}_{\ell}}\omega_{k,\ell} = 1$\\
            $\displaystyle q^{(\ell)} = \sum_{k \in \mathbb{K}_{\ell}}\omega_{k,\ell} p_{k}^{(\ell)} -  p^{(\ell)}$ \\
            $L_{\ell} =
            %\left \{
            \begin{cases}
                \displaystyle\frac{\sum_{k \in \mathbb{K}_{\ell}} \omega_{k,\ell}\nnorm{p_{k}^{(\ell)}-p^{(\ell)}}^{2}}{\nnorm{q^{(\ell)}}^{2}}
                \eqfinv
                & \textrm{if }
                p^{(\ell)} \notin \bigcap_{k \in \mathbb{K}_{\ell}} \mathcal{C}_{k}
                \\
                1
                \eqfinv
                & \textrm{otherwise}
            \end{cases}
            %\right .
            $\\
%         \end{enumerate}
         \emph{$r^{(\ell)} = p^{(\ell)} - L_{\ell}q^{(\ell)}$}\;
         $p^{(\ell+1)} = Q(p^{(0)},p^{(\ell)},r^{(\ell)})$\;
%       \emph{Set $\pi_{n} = \proscal{x_{0} - x_{n}}{x_{n}-z_{n}}$,
%            $\mu_{n} = \nnorm{x_{0} - x_{n}}^{2}$, $\nu_{n} = \nnorm{x_{n}-z_{n}}^{2}$
%                and $\rho_{n} = \mu_{n}\nu_{n} - \pi_{n}^{2}$} \;
%        \emph{Compute
%            $x_{n+1}
%            =
%            \left \{
%            \begin{cases}
%                z_{n}
%                \eqfinv
%                & \textrm{if }
%                \rho_{n} = 0 \textrm{ and } \pi_{n} \geq 0
%                \\
%                x_{0} + \vardelim{1+\frac{\pi_{n}}{\nu_{n}}}\np{z_{n}-x_{n}}
%                \eqfinv
%                & \textrm{if }
%                \rho_{n} > 0 \textrm{ and } \pi_{n}\nu_{n} \geq \rho_{n}
%                \\
%                x_{n} + \frac{\nu_{n}}{\rho_{n}}\bp{\pi_{n}\np{x_{0}-x_{n}}+\mu_{n}\np{z_{n}-x_{n}}}
%                \eqfinv
%                & \textrm{if }
%                \rho_{n} > 0 \textrm{ and } \pi_{n}\nu_{n} < \rho_{n}
%            \end{cases}
%            \right .
%            $
%        }\;
%        }
%        $p_{n+1} = x_{n}$ \;
    }
    \KwRet{$ u^{(n+1)}=p^{\rm end}$}
 \caption{Projection algorithm.}
\end{algorithm}
\DecMargin{1em}

\section{Application to robust binary classification}\label{se:sim}

\subsection{Context}
In this section, we illustrate the performance of our approach on different scenarios in the context of binary classification. To this aim, we consider the \texttt{ionosphere} and \texttt{colon-cancer} datasets~\footnote{\url{https://www.csie.ntu.edu.tw/~cjlin/libsvmtools/datasets/binary.html}.}. The respective numbers of observations $N$ and of features $d$ are summarized in Table~\ref{tab:dataset}. Unless specified, we will consider the original datasets without pre-processing, using a training set with 60\% of the original database and a testing set gathering the remaining entries. The splitting between training and testing samples is performed using function \texttt{train\_test\_split} of Scikit-learn~\footnote{\url{https://scikit-learn.org}}. We propose to compare the classical formulation in Equation \eqref{e:classform} with the formulation in Problem \ref{prop_reformulation_RM} (resp. Problem \ref{prop_reformulation_kuhn}) that uses ambiguity sets defined through the Kullback-Leibler divergence (resp. Wasserstein distance). We make use of the logistic regression loss in \eqref{e:logreg} (\cite{LogisticChierchia} for recent developments). The constrained minimization problems are solved running the proposed Algorithm~\ref{algo_projection} over a sufficient number of iterations so as to reach the stability criterion $\|p^{(\ell+1)}-p^{(\ell)}\| \leq 10^{-5}$. All the tests are performed by using Julia programming language, on a computer with processors Intel® Core™ i7-3610QM CPU @ 2.30GHz × 8 and 16Gb of RAM. 

\begin{table}[h]
    \centering
    \begin{tabular}{c|c|c}
        Name of dataset
        & \texttt{ionosphere}
        & \texttt{colon-cancer}
        \\
        \hline
        \hline
        Number of observations ($N$)
        &351
        &64
        \\
        \hline
        Number of features ($d$)
        &34
        &2000
    \end{tabular}
\caption{Parameters of the datasets.}
\label{tab:dataset}
\end{table}

\subsection{Results in standard conditions}
\label{section_convergence}

\subsubsection{Ionosphere dataset}

We display the evolution of the difference between the current cost function and its final value (computed after a very high number of iterations), with respect to the iteration number (see Figure~\ref{fig_convergence_iter}) and CPU time (see Figure \ref{fig_convergence_time}). In the case of the Kullback-Leibler divergence, we choose $\mathbb{K}_\ell = [1,K]$ while for Wasserstein distance, we set the cardinality of $\mathbb{K}_\ell$ equal to $1500$ ($K = 44100$ in this case). In this example, we observe that the convergence speed is slightly increased for large values of $\varepsilon$. Regarding the comparison between the two ambiguity sets, it can be observed on Figure \ref{fig_convergence_time} that the method is faster in the case of the Kullback-Leibler divergence since the number of constraints grows linearly as a function of the number of observations, whereas the growth is quadratic in the case of the Wasserstein distance. 

\begin{figure*}[h!]
    \centering
    \begin{subfigure}[t]{0.5\textwidth}
        \centering
        \includegraphics[scale = 0.5]{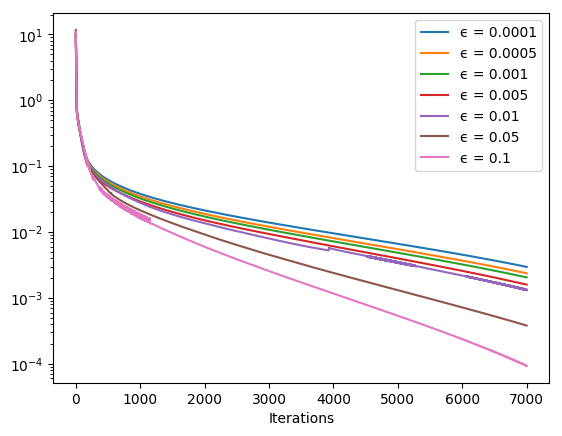}
        \caption{Kullback-Leibler divergence}
    \end{subfigure}%
    ~ 
    \begin{subfigure}[t]{0.5\textwidth}
        \centering
        \includegraphics[scale = 0.5]{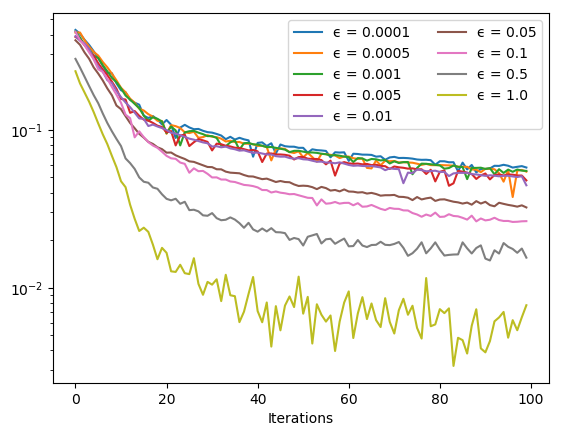}
        \caption{Wasserstein distance}
    \end{subfigure}
    \caption{\texttt{ionosphere} dataset: Log of the difference between current loss and final loss, with respect to the iteration number for various values of $\epsilon$.
    }
    \label{fig_convergence_iter}
\end{figure*}

\begin{figure*}[t!]
    \centering
    \begin{subfigure}[t]{0.5\textwidth}
        \centering
        \includegraphics[scale = 0.5]{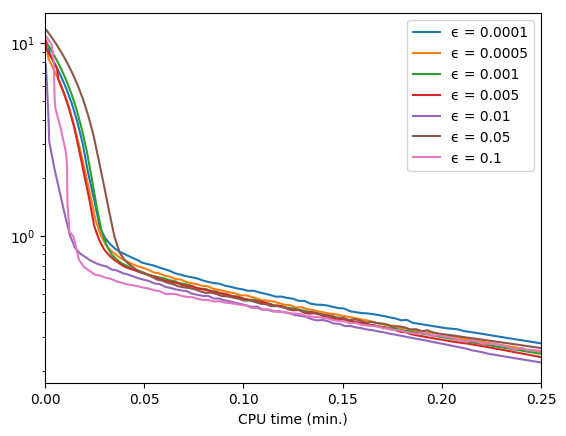}
        \caption{Kullback-Leibler divergence}
    \end{subfigure}%
    ~ 
    \begin{subfigure}[t]{0.5\textwidth}
        \centering
        \includegraphics[scale = 0.5]{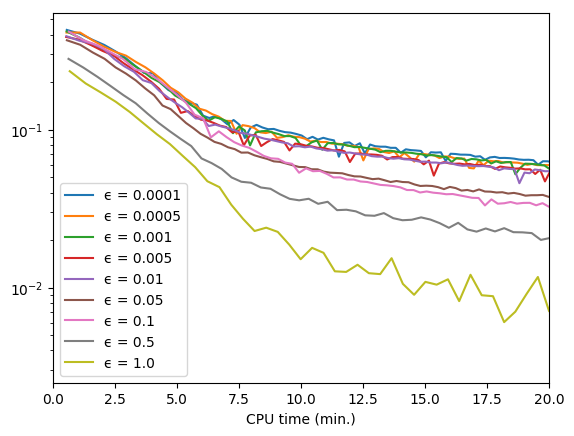}
        \caption{Wasserstein distance}
    \end{subfigure}
    \caption{\texttt{ionosphere} dataset: Log of the difference between current loss and final loss, with respect to the CPU time for vaious values of $\epsilon$ over the first 100 iterations.}
    \label{fig_convergence_time}
\end{figure*}

% \begin{figure}[!ht]
%     \centering
%     \includegraphics[scale = 0.7]{KLlogiter.png}
%     \caption{Log of the error 
%     with respect to iteration number $n$ for different values of $\epsilon$ 
%     with Kullback-Leibler divergence.}
%     \label{fig_conv_KL}
% \end{figure}

% \begin{figure}[!ht]
%     \centering
%     \includegraphics[scale = 0.7]{convergenceWAS.png}
%     \caption{Log of the error
%     with respect to iteration number $n$ for different values of $\epsilon$ with Wasserstein distance.}
%     \label{fig_conv_WAS}
% \end{figure}
% 
% \begin{figure}[!ht]
%     \centering
%     \includegraphics[scale = 0.7]{KLlogtime.png}
%     \caption{Log of the error
%     with respect to CPU time for different values of $\epsilon$ with Kullback-Leibler divergence.}
%     \label{fig_time_KL}
% \end{figure}
% 
% \begin{figure}[!ht]
%     \centering
% %     \includegraphics[scale = 0.7]{}
%     \caption{Log of the error  with respect to CPU time
%     for different values of $\epsilon$ with Wasserstein distance.}
%     \label{fig_time_WAS}
% \end{figure}

Figure~\ref{fig_AUC_epsilon} shows the value of the area under the ROC curve (AUC metric)~\cite{bradley1997use}
as a function of $\epsilon$, computed using the testing set. There is a clear compromise in the choice of the value of $\epsilon$ for maximizing the AUC, and the best performance are obtained for an intermediate non-zero value of this parameter. This clearly illustrates the benefit of the proposed formulation. Note that such results are consistent with the conclusions in~\citet*{shafieezadeh2015distributionally} and~\citet*{gotoh2018robust}. On this example, the Wassertein ambiguity set provides better results than the Kullback-Leibler divergence 
but it should be reminded that it comes at the expense of a higher computational cost.
% \begin{table}[!ht]
%     \centering
%     \begin{tabular}{c|c|c}
%         Value of $\epsilon$
%         & AUC with KL
%         & AUC with Wasserstein
%         \\
%         \hline
%         \hline
%         $\epsilon = 0$ (LR)
%         &0.746
%         &0.746
%         \\
%         $\epsilon = 0.001$
%         &0.773
%         &0.829
%         \\
%         $\epsilon = 0.002$
%         &\textbf{0.774}
%         &0.828
%         \\
%         $\epsilon = 0.003$
%         &0.772
%         &\textbf{0.836}
%         \\
%         $\epsilon = 0.004$
%         &0.771
%         &0.826
%         \\
%         $\epsilon = 0.005$
%         &0.770
%         &0.816
%         \\
%         $\epsilon = 0.006$
%         &0.770
%         &0.824
%         \\
%         $\epsilon = 0.007$
%         &0.769
%         &0.820
%         \\
%         $\epsilon = 0.008$
%         &0.767
%         &0.823
%         \\
%         $\epsilon = 0.009$
%         &0.767
%         &0.831
%         \\
%         $\epsilon = 0.01$
%         &0.756
%         &0.825
%         \\
%         $\epsilon = 0.05$
%         &0.742
%         &0.813
%         \\
%         $\epsilon = 0.1$
%         &0.718
%         &0.812
%     \end{tabular}
%     \caption{Values of the area under ROC curve for different value of $\epsilon$ for ionosphere dataset.}
%     \label{table_scor_io}
% \end{table}

\begin{figure*}[h!]
    \centering
    \begin{subfigure}[t]{0.5\textwidth}
        \centering
        \includegraphics[scale = 0.5]{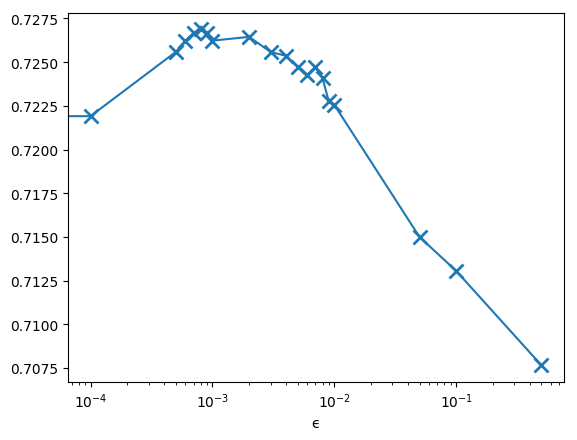}
        \caption{Kullback-Leibler divergence}
    \end{subfigure}%
    ~ 
    \begin{subfigure}[t]{0.5\textwidth}
        \centering
        \includegraphics[scale = 0.5]{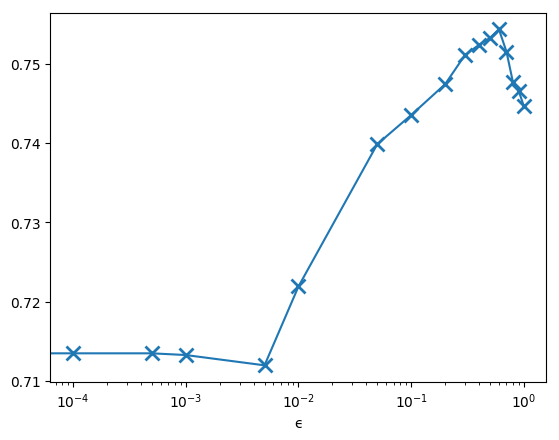}
        \caption{Wasserstein distance}
    \end{subfigure}
    \caption{\texttt{ionosphere} dataset: AUC metric as a function of $\epsilon$.}
    \label{fig_AUC_epsilon}
\end{figure*}

% \begin{figure}[!ht]
%     \centering
%     \includegraphics[scale = 0.7]{KLepsi.png}
%     \caption{Evolution of AUC in function of $\epsilon$ with Kullback-Leibler divergence.}
%     \label{KLepsi}
% \end{figure}
% 
% \begin{figure}[!ht]
%     \centering
%     \includegraphics[scale = 0.7]{WASepsi.png}
%     \caption{Evolution of AUC in function of $\epsilon$ with Wasserstein distance.}
%     \label{WASepsi}
% \end{figure}

\subsubsection{Colon-cancer dataset}
\label{section_colon_cancer}

We now present in Table~\ref{table_scor_colon_cancer} the evolution of the AUC for tests performed on the \texttt{colon-cancer} dataset. 
This dataset only contains 64 observations. For such small dataset, the formulation in Problem \ref{prop_reformulation_kuhn} becomes very cheap in terms of computational cost. As can be noticed in Table~\ref{table_scor_colon_cancer}, taking a nonzero value for $\epsilon$
leads to an increase of about 7\% in terms AUC, which is significant in such challenging context. These results allow to assess the robustness property of the proposed formulation. 

\begin{table}[!ht]
    \centering
    \begin{tabular}{c|c|c}
        Value of $\epsilon$
        & AUC with KL
        & AUC with Wasserstein
        \\
        \hline
        \hline
        $\epsilon = 0$ (LR)
        &0.832
        &0.832
        \\
        $\epsilon = 0.001$
        &0.757
        &0.787
        \\
        $\epsilon = 0.002$
        &0.750
        &0.770
        \\
        $\epsilon = 0.003$
        &0.779
        &0.706
        \\
        $\epsilon = 0.004$
        &0.698
        &0.691
        \\
        $\epsilon = 0.005$
        &0.868
        &0.831
        \\
        $\epsilon = 0.006$
        &\textbf{0.890}
        &0.860
        \\
        $\epsilon = 0.007$
        &0.728
        &0.838
        \\
        $\epsilon = 0.008$
        &0.809
        &0.768
        \\
        $\epsilon = 0.009$
        &0.875
        &\textbf{0.890}
        \\
        $\epsilon = 0.01$
        &0.801
        &0.853
        \\
        $\epsilon = 0.05$
        &0.786
        &0.794
        \\
        $\epsilon = 0.1$
        &0.801
        &0.816
    \end{tabular}
    \caption{\texttt{colon-cancer} dataset: Values of the AUC for different values of $\epsilon$.}
    \label{table_scor_colon_cancer}
\end{table}

\subsection{Results on an altered database}
\label{section_distribution}

Let us come back to the processing of \texttt{ionosphere} dataset. In order to better illustrate the interest of the proposed formulation, we propose to modify the training set so that the proportion of labels $(-1)$ and $(+1)$ is altered and unbalanced. Such situation could
typically arise during a transient regime, such as the beginning of an epidemic, or in the case of an incomplete dataset.
After dividing the dataset between a training set and a testing set, using the same 60\% ratio as in our previous tests, we will drop randomly a certain number of observations associated with the label $(-1)$ so that the proportion of this label becomes ten times lower than its original proportions. Figure~\ref{fig_ROC_curve} displays ROC curves and Table~\ref{table_mod_io} evaluates AUC metric, for various values of $\epsilon$. Our formulation clearly outperforms the classical logistic regression classifier (retrieved when $\epsilon = 0)$.
Noticeably, the later presents the same area under the curve as a random classifier, and thus exhibits a similar behavior to such a classifier.

\begin{figure*}[h!]
    \centering
    \begin{subfigure}[t]{0.5\textwidth}
        \centering
        \includegraphics[scale = 0.5]{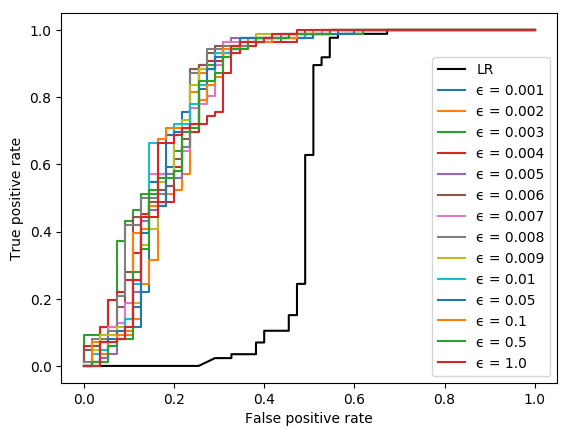}
        \caption{Kullback-Leibler divergence}
    \end{subfigure}%
    ~ 
    \begin{subfigure}[t]{0.5\textwidth}
        \centering
        \includegraphics[scale = 0.5]{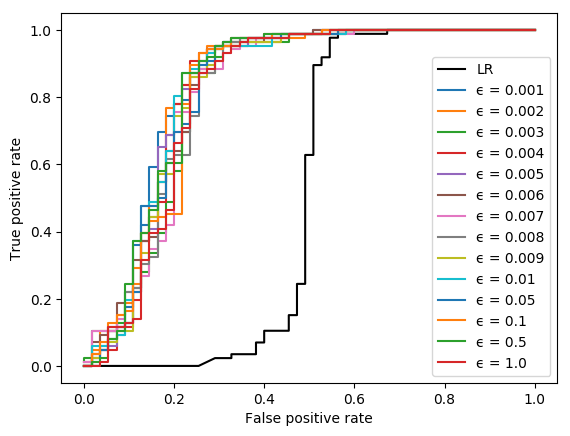}
        \caption{Wasserstein distance}
    \end{subfigure}
    \caption{\texttt{ionosphere} dataset (altered): ROC curve for different values of $\epsilon$.}
    \label{fig_ROC_curve}
\end{figure*}

% \begin{figure}[!ht]
%     \centering
%     \includegraphics[scale = 0.7]{AUCKLmod.png}
%     \caption{ROC curve for different values of $\epsilon$ with Kullback-Leibler divergence.}
%     \label{fig_mod_KL}
% \end{figure}
% 
% \begin{figure}[!ht]
%     \centering
%     \includegraphics[scale = 0.7]{AUCWASmod.png}
%     \caption{ROC for different values of $\epsilon$ with Wasserstein distance.}
%     \label{fig_mod_WAS}
% \end{figure}

\begin{table}[!ht]
    \centering
    \begin{tabular}{c|c|c}
        Value of $\epsilon$
        & AUC with KL
        & AUC with Wasserstein
        \\
        \hline
        \hline
        $\epsilon = 0$ (LR)
        &0.514
        &0.514
        \\
        $\epsilon = 0.001$
        &0.816
        &\textbf{0.840}
        \\
        $\epsilon = 0.002$
        &0.804
        &0.835
        \\
        $\epsilon = 0.003$
        &\textbf{0.840}
        &0.814
        \\
        $\epsilon = 0.004$
        &0.824
        &0.830
        \\
        $\epsilon = 0.005$
        &0.815
        &0.829
        \\
        $\epsilon = 0.006$
        &0.834
        &0.829
        \\
        $\epsilon = 0.007$
        &0.821
        &0.815
        \\
        $\epsilon = 0.008$
        &0.835
        &0.815
        \\
        $\epsilon = 0.009$
        &0.823
        &0.822
        \\
        $\epsilon = 0.01$
        &0.828
        &0.835
        \\
        $\epsilon = 0.05$
        &0.815
        &0.826
        \\
        $\epsilon = 0.1$
        &0.824
        &0.823
    \end{tabular}
    \caption{\texttt{ionosphere} dataset (altered): Values of the area under ROC curve for different values of $\epsilon$.}
    \label{table_mod_io}
\end{table}

\subsection{Variance reduction study}\label{subsection_var_reduction}
\label{section_histogramme}
In a nutshell, the robust framework based on the Wasserstein distance provides a better expected reward but at the expense of a higher computational cost. The risk measure based on the Kullback-Leibler divergence is easily tractable, provides a reduction of the variance in out-of-samples results, but a smaller increase in terms of expected reward (see ~\citet*{gotoh2018robust} for a more detailed theoretical analysis).
In practice, as discussed in ~\citet*{shafieezadeh2015distributionally} and in~\citet*{gotoh2018robust},
``a little of robustness'' typically improves a bit the expected
reward (around 1\%), however results in a larger reduction in terms of variance.
We propose to reproduce such an analysis by means of two experiments using \texttt{ionosphere} dataset. 
We first consider the case when a small training set is used where only 10\% of the data are available. Then we focus on the case when 60\% of the data are used as training set. In both cases, 1000 random realizations are run, when we solve
the classical formulation using logistic regression (LR) loss in Equation \eqref{e:classform},
the formulation in Problem \ref{prop_reformulation_RM}
that uses ambiguity sets defined through the Kullback-Leibler divergence, 
and the formulation in Problem \ref{prop_reformulation_kuhn} 
that uses ambiguity sets defined through the Wasserstein distance.
We then compute the value of the AUC metric on the associated testing set and display results as histograms. 
When we use 10\% of the data for training (Figure~\ref{fig_10_percent}), we see the benefits of our robust solution with respect to the standard LR  classifier. When more data are collected, the probability distribution becomes more accurate and our robust models tend to produce the same outputs as when using classical logistic regression (Figure~\ref{fig_60_percent}).

\begin{figure*}[t!]
    \centering
    \begin{subfigure}[t]{0.5\textwidth}
        \centering
        \includegraphics[scale = 0.5]{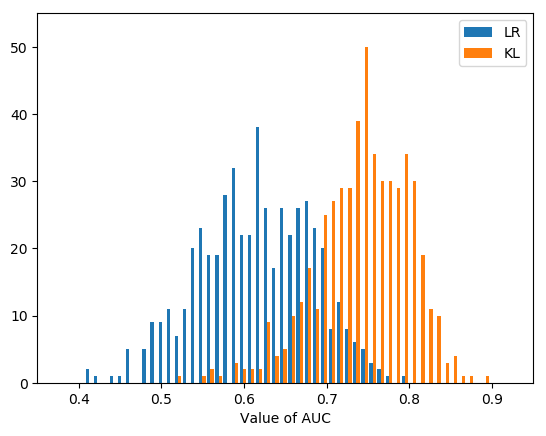}
        \caption{Kullback-Leibler divergence}
    \end{subfigure}%
    ~ 
    \begin{subfigure}[t]{0.5\textwidth}
        \centering
        \includegraphics[scale = 0.5]{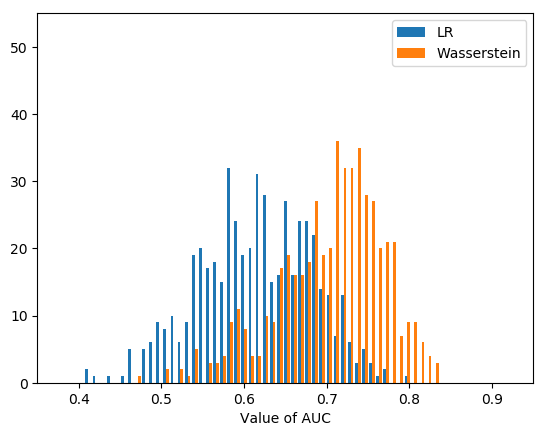}
        \caption{Wasserstein distance}
    \end{subfigure}
    \caption{\texttt{ionosphere} dataset: AUC histogram for 1000 random realizations using 10\% of data for the training set. Robust model is used with $\epsilon = 0.001$.}
    \label{fig_10_percent}
\end{figure*}

% \begin{figure}[!ht]
%     \centering
%     \includegraphics[scale = 0.7]{var10pcKL.png}
%     \caption{AUC for 1000 simulations with 10\% of data for the training set for logistic regression and robust model with Kullback Leibler divergence and $\epsilon = 0.001$.}
%     \label{fig_var_KL}
% \end{figure}
% 
% \begin{figure}[!ht]
%     \centering
%     \includegraphics[scale = 0.7]{var10pcWAS.png}
%     \caption{AUC for 1000 simulations with 10\% of data for the training set for logistic regression and robust model with Kullback Leibler divergence and $\epsilon = 0.001$.}
%     \label{fig_var_WAS}
% \end{figure}

\begin{figure*}[t!]
    \centering
    \begin{subfigure}[t]{0.5\textwidth}
        \centering
        \includegraphics[scale = 0.5]{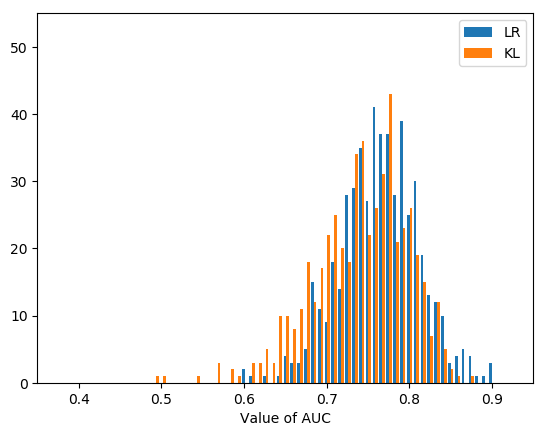}
        \caption{Kullback-Leibler divergence}
    \end{subfigure}%
    ~ 
    \begin{subfigure}[t]{0.5\textwidth}
        \centering
        \includegraphics[scale = 0.5]{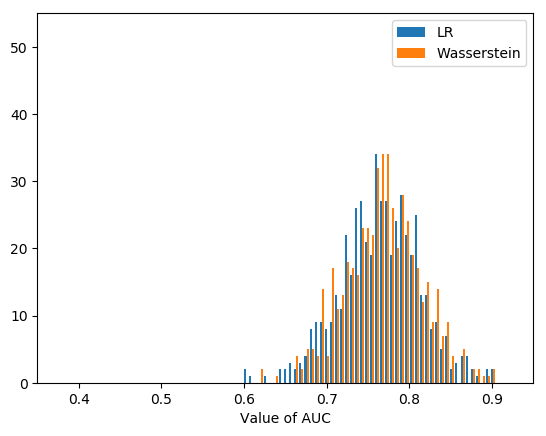}
        \caption{Wasserstein distance}
    \end{subfigure}
    \caption{\texttt{ionosphere} dataset: AUC histogram for 1000 random realizations using 60\% of data for the training set. Robust model is used with $\epsilon = 0.001$.
    }
    \label{fig_60_percent}
    
\end{figure*}

\section{Conclusion} \label{se:conclu}
We have highlighted that risk measures offer versatile tools for addressing 
machine learning problems in a robust manner.
By assuming that the loss function is convex, the related optimization problem
has been recast as a convex one. We have shown that various classes of risk measures, e.g. those based on $\varphi$-divergences
or on the Wasserstein distance, lead to a common convex formulation. In addition, an efficient 
convex optimization algorithm has been proposed to cope with the non trivial constrained problem resulting from this formulation.
We have conducted numerical experiments in which various ambiguity sets are tackled  thanks to the same algorithm. We have also illustrated 
that the considered robust models can outperform classical ones in challenging contexts when the
size of the training set is limited, or when the distribution of labels in the training set is not representative of the reality.

\section*{Acknowledgments} 
The second author wants to thank Université Paris-Est and Labex Bézout for the financial support
particularly for the funding of his PhD program. The work of J.-C. Pesquet was supported by Institut 
Universitaire de France.


\begin{thebibliography}{35}
\providecommand{\natexlab}[1]{#1}
\providecommand{\url}[1]{\texttt{#1}}
\expandafter\ifx\csname urlstyle\endcsname\relax
  \providecommand{\doi}[1]{doi: #1}\else
  \providecommand{\doi}{doi: \begingroup \urlstyle{rm}\Url}\fi

\bibitem[Ali and Silvey(1966)]{ali1966general}
S.~M. Ali and S.~D. Silvey.
\newblock A general class of coefficients of divergence of one distribution
  from another.
\newblock \emph{Journal of the Royal Statistical Society. Series B
  (Methodological)}, pages 131--142, 1966.

\bibitem[Artzner et~al.(1999)Artzner, Delbaen, Eber, and
  Heath]{artzner1999coherent}
P.~Artzner, F.~Delbaen, J.-M. Eber, and D.~Heath.
\newblock Coherent measures of risk.
\newblock \emph{Mathematical finance}, 9\penalty0 (3):\penalty0 203--228, 1999.

\bibitem[Basseville(2013)]{basseville2013divergence}
M.~Basseville.
\newblock Divergence measures for statistical data processing—an annotated
  bibliography.
\newblock \emph{Signal Processing}, 93\penalty0 (4):\penalty0 621--633, 2013.

\bibitem[Bauschke and Combettes(2011)]{bouquin_combettes}
H.~H. Bauschke and P.~L. Combettes.
\newblock \emph{Convex Analysis and Monotone Operator Theory in Hilbert
  Spaces}.
\newblock Springer, New York, 2011.

\bibitem[Beck and Teboulle(2009)]{Beck-Teboule}
A.~Beck and M.~Teboulle.
\newblock A fast iterative shrinkage-thresholding algorithm for linear inverse
  problems.
\newblock \emph{SIAM Journal on Imaging Sciences}, 2\penalty0 (1):\penalty0
  183--202, 2009.

\bibitem[Ben-Tal and Nemirovski(2000)]{ben2000robust}
A.~Ben-Tal and A.~Nemirovski.
\newblock Robust solutions of linear programming problems contaminated with
  uncertain data.
\newblock \emph{Mathematical programming}, 88\penalty0 (3):\penalty0 411--424,
  2000.

\bibitem[Ben-Tal et~al.(2009)Ben-Tal, El~Ghaoui, and Nemirovski]{ben2009robust}
A.~Ben-Tal, L.~El~Ghaoui, and A.~Nemirovski.
\newblock \emph{Robust optimization}.
\newblock Princeton University Press, 2009.

\bibitem[Ben-Tal et~al.(2013)Ben-Tal, Den~Hertog, De~Waegenaere, Melenberg, and
  Rennen]{ben2013robust}
A.~Ben-Tal, D.~Den~Hertog, A.~De~Waegenaere, B.~Melenberg, and G.~Rennen.
\newblock Robust solutions of optimization problems affected by uncertain
  probabilities.
\newblock \emph{Management Science}, 59\penalty0 (2):\penalty0 341--357, 2013.

\bibitem[Bradley(1997)]{bradley1997use}
A.~P. Bradley.
\newblock The use of the area under the roc curve in the evaluation of machine
  learning algorithms.
\newblock \emph{Pattern recognition}, 30\penalty0 (7):\penalty0 1145--1159,
  1997.

\bibitem[Briceno-Arias et~al.(2017)Briceno-Arias, Chierchia, Chouzenoux, and
  Pesquet]{LogisticChierchia}
L.~M. Briceno-Arias, G.~Chierchia, E.~Chouzenoux, and J.-C. Pesquet.
\newblock A random block-coordinate douglas-rachford splitting method with low
  computational complexity for binary logistic regression.
\newblock \emph{arXiv preprint arXiv:1712.09131}, 2017.

\bibitem[Chambolle and Dossal(2015)]{Chambolle-Dossal}
A.~Chambolle and C.~Dossal.
\newblock \emph{Journal of Optimization Theory and Applications}, 166\penalty0
  (3):\penalty0 968--982, 2015.

\bibitem[Combettes(2000)]{outerCombettes2000}
P.~L. Combettes.
\newblock Strong convergence of block-iterative outer approximation methods for
  convex optimization.
\newblock \emph{SIAM Journal on Control and Optimization}, 38\penalty0
  (2):\penalty0 538--565, 2000.

\bibitem[Combettes(2003)]{combettes2003block}
P.~L. Combettes.
\newblock A block-iterative surrogate constraint splitting method for quadratic
  signal recovery.
\newblock \emph{IEEE Transactions on Signal Processing}, 51\penalty0
  (7):\penalty0 1771--1782, 2003.

\bibitem[Combettes and M{\"u}ller(2018)]{combettes2018perspective}
P.~L. Combettes and C.~L. M{\"u}ller.
\newblock Perspective functions: Proximal calculus and applications in
  high-dimensional statistics.
\newblock \emph{Journal of Mathematical Analysis and Applications},
  457\penalty0 (2):\penalty0 1283--1306, 2018.

\bibitem[Combettes and Pesquet(2010)]{articletutoriel-Combettes-Pesquet}
P.~L. Combettes and J.-C. Pesquet.
\newblock Proximal splitting methods in signal processing.
\newblock In H.~H. Bauschke, R.~Burachik, P.~L. Combettes, V.~Elser, D.~R.
  Luke, and H.~Wolkowicz, editors, \emph{Fixed-Point Algorithms for Inverse
  Problems in Science and Engineering}, pages 185--212. Springer-Verlag, New
  York, 2010.

\bibitem[Combettes et~al.(2010)Combettes, Dung, and V\~u]{CombettesDFB}
P.~L. Combettes, D.~Dung, and B.~C. V\~u.
\newblock Dualization of signal recovery problems.
\newblock \emph{Set-Valued and Variational Analysis}, 18\penalty0 (3):\penalty0
  373--404, Dec. 2010.

\bibitem[Csisz{\'a}r(1964)]{csiszar1964informationstheoretische}
I.~Csisz{\'a}r.
\newblock Eine informationstheoretische ungleichung und ihre anwendung auf
  beweis der ergodizitaet von markoffschen ketten.
\newblock \emph{Magyer Tud. Akad. Mat. Kutato Int. Koezl.}, 8:\penalty0
  85--108, 1964.

\bibitem[Duchi et~al.(2016)Duchi, Glynn, and Namkoong]{duchi2016statistics}
J.~Duchi, P.~Glynn, and H.~Namkoong.
\newblock Statistics of robust optimization: A generalized empirical likelihood
  approach.
\newblock \emph{arXiv preprint arXiv:1610.03425}, 2016.

\bibitem[Esfahani and Kuhn(2015)]{esfahani2015data}
P.~M. Esfahani and D.~Kuhn.
\newblock Data-driven distributionally robust optimization using the
  wasserstein metric: Performance guarantees and tractable reformulations.
\newblock \emph{arXiv preprint arXiv:1505.05116}, 2015.

\bibitem[Esfahani et~al.(2017)Esfahani, Shafieezadeh-Abadeh, Hanasusanto, and
  Kuhn]{esfahani2017data}
P.~M. Esfahani, S.~Shafieezadeh-Abadeh, G.~A. Hanasusanto, and D.~Kuhn.
\newblock Data-driven inverse optimization with imperfect information.
\newblock \emph{Mathematical Programming}, pages 1--44, 2017.

\bibitem[Feng et~al.(2014)Feng, Xu, Mannor, and Yan]{feng2014robust}
J.~Feng, H.~Xu, S.~Mannor, and S.~Yan.
\newblock Robust logistic regression and classification.
\newblock In \emph{Advances in neural information processing systems}, pages
  253--261, 2014.

\bibitem[F{\"o}llmer and Schied(2016)]{follmer2016stochastic}
H.~F{\"o}llmer and A.~Schied.
\newblock \emph{Stochastic finance: an introduction in discrete time (4th
  edition)}.
\newblock Walter de Gruyter, 2016.

\bibitem[Gotoh et~al.(2018)Gotoh, Kim, and Lim]{gotoh2018robust}
J.-y. Gotoh, M.~J. Kim, and A.~E. Lim.
\newblock Robust empirical optimization is almost the same as mean--variance
  optimization.
\newblock \emph{Operations Research Letters}, 2018.

\bibitem[Haugazeau(1968)]{haugazeau1968inequations}
Y.~Haugazeau.
\newblock Sur les in{\'e}quations variationnelles et la minimisation de
  fonctionnelles convexes.
\newblock \emph{These, Universite de Paris}, 1968.

\bibitem[Hu and Hong(2013)]{hu2013kullback}
Z.~Hu and L.~J. Hong.
\newblock Kullback-leibler divergence constrained distributionally robust
  optimization.
\newblock \emph{Available at Optimization Online}, 2013.

\bibitem[Kurakin et~al.(2016)Kurakin, Goodfellow, and
  Bengio]{kurakin2016adversarial}
A.~Kurakin, I.~Goodfellow, and S.~Bengio.
\newblock Adversarial examples in the physical world.
\newblock \emph{arXiv preprint arXiv:1607.02533}, 2016.

\bibitem[Moghaddam and Mahlooji(2016)]{moghaddam2016robust}
S.~Moghaddam and M.~Mahlooji.
\newblock Robust simulation optimization using $\varphi$-divergence.
\newblock \emph{International Journal of Industrial Engineering Computations},
  7\penalty0 (4):\penalty0 517--534, 2016.

\bibitem[Morimoto(1963)]{morimoto1963markov}
T.~Morimoto.
\newblock Markov processes and the h-theorem.
\newblock \emph{Journal of the Physical Society of Japan}, 18\penalty0
  (3):\penalty0 328--331, 1963.

\bibitem[Namkoong and Duchi(2016)]{namkoong2016stochastic}
H.~Namkoong and J.~C. Duchi.
\newblock Stochastic gradient methods for distributionally robust optimization
  with f-divergences.
\newblock In \emph{Advances in Neural Information Processing Systems}, pages
  2208--2216, 2016.

\bibitem[Papernot et~al.(2016)Papernot, McDaniel, and
  Goodfellow]{papernot2016transferability}
N.~Papernot, P.~McDaniel, and I.~Goodfellow.
\newblock Transferability in machine learning: from phenomena to black-box
  attacks using adversarial samples.
\newblock \emph{arXiv preprint arXiv:1605.07277}, 2016.

\bibitem[Plan and Vershynin(2013)]{plan2013robust}
Y.~Plan and R.~Vershynin.
\newblock Robust 1-bit compressed sensing and sparse logistic regression: A
  convex programming approach.
\newblock \emph{IEEE Transactions on Information Theory}, 59\penalty0
  (1):\penalty0 482--494, 2013.

\bibitem[Rockafellar and Uryasev(2000)]{rockafellar2000optimization}
R.~T. Rockafellar and S.~Uryasev.
\newblock Optimization of conditional value-at-risk.
\newblock \emph{Journal of risk}, 2:\penalty0 21--42, 2000.

\bibitem[Ruszczy{\'n}ski and Shapiro(2006)]{ruszczynski2006conditional}
A.~Ruszczy{\'n}ski and A.~Shapiro.
\newblock Conditional risk mappings.
\newblock \emph{Mathematics of operations research}, 31\penalty0 (3):\penalty0
  544--561, 2006.

\bibitem[Ruszczynski and Shapiro(2006)]{ruszczynski2006optimization}
A.~Ruszczynski and A.~Shapiro.
\newblock Optimization of convex risk functions.
\newblock \emph{Mathematics of operations research}, 31\penalty0 (3):\penalty0
  433--452, 2006.

\bibitem[Shafieezadeh-Abadeh et~al.(2015)Shafieezadeh-Abadeh, Esfahani, and
  Kuhn]{shafieezadeh2015distributionally}
S.~Shafieezadeh-Abadeh, P.~M. Esfahani, and D.~Kuhn.
\newblock Distributionally robust logistic regression.
\newblock In \emph{Advances in Neural Information Processing Systems}, pages
  1576--1584, 2015.

\end{thebibliography}
\end{document}